\documentclass[11pt,reqno]{article}
\usepackage[top = 2.5cm, bottom = 2.5cm, left = 2.5cm, right = 2.5cm]{geometry}
\usepackage[usenames,dvipsnames]{color}

 \usepackage{theoremref}
  \usepackage{enumerate}

\newcommand{\lab}[1]{\label{#1}}               
 \newcommand{\thlab}[1]{\thlabel{#1}} 

\usepackage{amssymb}
\usepackage{amsfonts,mathrsfs}
\usepackage{amsmath,bbm,amsthm}
\usepackage{graphicx, mathabx}

\newcommand{\bee}{\begin{equation}}
\newcommand{\ee}{\end{equation}}
\newcommand{\bea}{\begin{eqnarray}}
\newcommand{\eea}{\end{eqnarray}}

\newcommand{\bean}{\begin{eqnarray*}}
\newcommand{\eean}{\end{eqnarray*}}
\newcommand\eqn[1]{(\ref{#1})}
\newcommand{\bel}[1]{\bee\lab{#1}}
 \catcode`@=11 \@addtoreset{equation}{section}

\catcode`@=12
  \def\qed{~~\vrule height8pt width4pt depth0pt}

\renewcommand\hat[1]{\widehat{#1}}
\renewcommand\bar[1]{\widebar{#1}}
\newcommand\se{\subseteq}
\newcommand\sm{\setminus}
\newcommand{\bad}{\ensuremath {\mathbf{bad}}}

\newcommand{\Bad}{\ensuremath {B}}
  
\def\Gnp{{\cal G}(n,p)}
\def\eps{\varepsilon}
\newcommand{\G}{\mathcal{G}}
\newcommand{\cG}{\mathcal{G}}
\newcommand{\cB}{\mathcal{B}}
\newcommand{\cS}{\mathcal{S}}
\newcommand{\Gnm}{\mathcal{G}(n,m)}
\def\Num{\mathcal{N}}

\def\Z{\mathbbm{Z}}

\def\Yc{{\cal Y}}

\def\YPavb{{Y}}
\def\yv{{\bf y}}

\def\Pc{{\cal P}}
\def\cP{{\cal P}}

\def\Rc{{\cal R}}
\def\Cc{{\cal C}}
\def\cC{{\cal C}}
\def\cA{{\cal A}}
\def\oA{{\vec{\cA}}}		
\def\D{{\mathfrak{D}}}	
\def\cD{{\cal D}}		
\def\R{\mathbbm{R}}		
\def\W{{\mathfrak{W}}} 	
\def\pr{{\bf P}}			
\def\Pr{\pr}			
\def\Var{{\bf Var}}

\def\ex{{\bf E}}
\def\dv{{\bf d}}
\def\dw{{\bf d}'}
\def\pv{{\bf p}}
\def\rv{{\bf r}}

\def\ve{{\bf e}}
\def\eb{{\ve_b}}
\def\ea{{\ve_a}}
\def\ev{{\ve_v}}
\def\Z{\mathbbm{Z}_{\ge 0}}

\def\Pgr{P^{\mathrm{gr}}}
\def\Rgr{R^{\mathrm{gr}}}
\def\Ygr{Y^{\mathrm{gr}}}

\def\Epp{{{\cal E}'_p}}
\def\Ep{{{\cal E}_p}}
\def\gtwo{{\gamma_2}}

\usepackage{mathtools}

\makeatletter
\DeclareRobustCommand\widecheck[1]{{\mathpalette\@widecheck{#1}}}
\def\@widecheck#1#2{%
    \setbox\z@\hbox{\m@th$#1#2$}%
    \setbox\tw@\hbox{\m@th$#1%
       \widehat{%
          \vrule\@width\z@\@height\ht\z@
          \vrule\@height\z@\@width\wd\z@}$}%
    \dp\tw@-\ht\z@
    \@tempdima\ht\z@ \advance\@tempdima2\ht\tw@ \divide\@tempdima\thr@@
    \setbox\tw@\hbox{%
       \raise\@tempdima\hbox{\scalebox{1}[-1]{\lower\@tempdima\box
\tw@}}}%
    {\ooalign{\box\tw@ \cr \box\z@}}}
\makeatother

\def\epsV{\varepsilon}
\def\epsA{\varepsilon}
\def\sigmaA{\sigma}
\def\sigmaB{\sigma}

\newcommand{\sigmaASquared}{\sigmaA\raisebox{0mm}{$^2$}}
\newcommand{\sigmaBSquared}{\sigmaB\raisebox{0mm}{$^2$}}

\newcommand{\size}[1]{\ensuremath{\left| #1 \right|}}

\newcommand{\reals}{\R}
\newtheorem{thm}{Theorem}[section]
\newtheorem{cor}[thm]{Corollary}
\newtheorem{conj}[thm]{Conjecture}
\newtheorem{lemma}[thm]{Lemma}
\newtheorem{proposition}[thm]{Proposition}

\newtheorem{claim}[thm]{Claim}

\newtheorem{definition}[thm]{Definition}

\date{}

\begin{document}
\title{Asymptotic enumeration of graphs by degree sequence, and the degree sequence of a random graph}

 \author{Anita Liebenau\thanks{Previously at Monash University where this research was carried out. Part of this research was supported by a DECRA Fellowship from the Australian Research Council.}\\
 {\small School of Mathematics and Statistics}\\{\small UNSW Sydney NSW 2052}\\
 {\small Australia} \\
 {\small  \tt{a.liebenau@unsw.edu.au}}
 \and Nick Wormald\thanks{Supported by an ARC Australian Laureate Fellowship.}\\
 {\small School of Mathematical Sciences}\\{\small Monash University VIC 3800}\\
 {\small Australia} \\
{\small  \tt{ nick.wormald@monash.edu}}
 }

 \maketitle
  
 \begin{abstract}
 In this paper we relate a fundamental parameter of a random graph, its  degree sequence, to a simple model of nearly independent binomial random variables. As a result, many interesting functions of the joint distribution of graph degrees, such as the distribution of the median degree, become amenable to estimation. Our result is established by proving an asymptotic formula conjectured in 1990 for the number of graphs with given degree sequence. In particular, this gives an asymptotic  formula for the number of $d$-regular graphs for all $d$, as $n\to\infty$. The key to our results is a new approach to estimating ratios between point probabilities  in the space of degree sequences of the random graph, including analysis of fixed points of the associated operators.
 \end{abstract}

\section{Introduction}\lab{s:intro}

We consider the number of graphs with a given degree sequence. In particular, we show  that a formula known to hold in the sparse and dense cases as long as the degrees of vertices are somewhat close to each other, also holds for the remaining cases. The main consequence of this is a very simple and consequently useful model for the degree sequence of a random graph.  The two most popularly studied models of random graphs are considered here: $\G(n,p)$, in which $n$ vertices have edges included  between each pair of them independently with probability $p$ for each pair, and $\G(n,m)$ in which $n$ vertices have $m$ edges included, chosen uniformly at random from the $m$-subsets of the unordered pairs of vertices. Those  classical models of random graphs easily satisfy the required restrictions on degrees  with high probability, and it follows that the degree sequence of  $\Gnm$ is well approximated by a certain sequence of independent binomial variables conditioned on summing to $2m$. A similar  connection is provided between the random graph $\Gnp$ and a slightly twisted sequence of independent binomial random variables. This makes a very convenient way of proving results about the degree sequence of $\Gnp$. Enumeration formulae for graphs by degree sequence have also led by other routes to a large number of results in random graph theory, some of which we mention in Section~\ref{s:further}.

The study of the degree sequence of the random graph goes back to the early papers of Erd{\H o}s and R{\'e}nyi~\cite{er1959,er1964} and has since attracted the attention of many, see, e.g.,~\cite{ba1982,bo1982,p1984,p1987,kr1987,bkr1989}.  For a historical overview, see for example Bollob{\' a}s' seminal book~\cite{Bo}, where the degree sequence is the first major topic. 
The distribution of the $k^{\mathrm{th}}$ largest element $d_k$ of the sequence, for example, was determined quite precisely when   $k$  is small. The book~\cite{BHJ} by Barbour, Holst and Janson contained much information on the distribution of the number $D_k$ of vertices of degree $k$.  
On the other hand, 
enumerating graphs with given degree sequence has been the interest of various authors over many years. 
Read \cite{read1958} found a recursive formula for the number of 3-regular graphs from which  he also deduced a simple asymptotic formula. After this, formulae for the number of graphs with given   degree sequence $\dv=(d_1,\ldots,d_n)$  were found by Bender and Canfield~\cite{bc1978} and Wormald~\cite[Theorem 3.3]{Wthesis}, and for ever-denser ranges of degrees by Bollob\'as \cite{b1980}, and McKay~\cite{Mc}, culminating in papers giving asymptotic formulae for a  range of degrees, provided the average degree $d$ is $o(\sqrt n)$ (by McKay and Wormald~\cite{MWlow}) or between $c n / \log n$ and $n/2$ for a certain $c$  (by  McKay and Wormald~\cite{MWhigh}, also treated more recently by Barvinok and Hartigan~\cite{BH} for a wider spread of degrees, but similar density).  The complementary ranges of $d$ larger than $n/2$ are automatically covered. Also quite recently, Janson~\cite{j2009,j2014} obtained formulae for some sparse degree sequences with maximum degree $\Theta(\sqrt n)$, and Gao and Wormald~\cite{GW} for others having slightly larger maximum degree.
A  case of special interest, which saw no advance since 1990, is the  problem of finding the asymptotic number of $d$-regular graphs for $d$ in the range $c\sqrt{n} \le d=o(n/\log n)$.

In 1990, McKay and Wormald~\cite{MWhigh} restated the asymptotic formulae from the sparse and dense cases in a common form, which they conjectured to be valid additionally for all densities in between those two cases and hence for all densities except trivial extremely sparse and dense ones. For a precise statement, see the {\em Binomial Approximation Conjecture} (Conjecture~\ref{conj1}) below. It applies  to all the typical degree sequences in either model of random graphs defined above. 
In this paper, we prove the Binomial Approximation Conjecture. 
In particular, as a very special case, this implies that the number of $d$-regular graphs on $n$ vertices is 
asymptotically equal to 
$$\frac{\displaystyle \binom{n-1}{d}^n  \binom{\binom{n}{2}}{m}}{\displaystyle\binom{n(n-1)}{2m}} \cdot e^{1/4}$$
for all $1\le d\le n-2$, where $m=dn/2$.

A weakened version of the Binomial Approximation Conjecture was given in 1997  by McKay and Wormald~\cite{degseq1} (see  \thref{conj2}) and shown to imply an explicit connection between the degree sequence of a random graph of a given density, and a sequence of independent binomial variables. 
 This work opened up a completely new approach to deriving properties of the degree sequence of a random graph, by considering independent binomials. It permits easy access to all the known properties, and a large number that   were previously inaccessible. Previously, precise results for order statistics were only known for extreme degrees, but now even the distribution of the median degree can be closely examined for those densities where the conjecture holds.

To present the connection between the degree sequence of $\Gnm$ or $\Gnp$, and a sequence of independent binomials, let us  first make some definitions. We assume that a graph  on $n$ vertices has vertex set $v_1,\ldots, v_n$ and degree sequence $(d_1,\ldots, d_n)$,  so that $ d(v_i)=d_i$.  If $\G$ is any probability space of random  graphs, let    $\cD(\G)$ be the random vector distributed as the degree sequence of a random graph $G\in\G$. Also define ${\cal B}_p(n)$ to be the random sequence consisting of $n$ independent binomial variables ${\rm Bin}(n-1,p)$. 

 Let $A_n$ and $B_n$ be two sequences of probability spaces  with the same underlying set for each $n$. Suppose that 
 whenever a sequence of events $H_n$ satisfies $\pr(H_n)=n^{-O(1)}$  in either model, 
it is true that $\pr_{A_n}(  H_n )\sim \pr_{B_n}(  H_n )$, where by $f(n)\sim g(n)$ we mean that $f(n)/g(n)\to 1$ as $n\to\infty$. We then say that $A_n$ and $B_n$ are {\em asymptotically quite equivalent} (a.q.e.). 
Throughout this paper  we use    $\omega $ to  be  an arbitrary function of $n$ such that $\omega \to\infty$ as $n\to\infty$, perhaps different at each occurrence. 
As we will see later, our main result, combined with existing results, implies the following. Note for part (i)  that $B_p(n)\mid_{\Sigma=2m}$ is independent of $p$.
 \begin{proposition}\thlab{p:model}
  Let $n,m$ be integers and let $0<p<1$.  
 Let $\Sigma$ denote  the sum of the components of the random vector ${\cal B}_p(n)$ in (i) and   $\cB_{\hat p}(n)$ in (ii). 
\begin{itemize}
\item[$(i)$]  $\cD(\Gnm)$  and $\cB_p(n)\mid_{\Sigma=2m}$ are a.q.e.\ provided that $ \min \{m, {n\choose 2}-m\}=\omega \log n$. 
\item[$(ii)$]  Let  $\hat{p}$ be randomly chosen according to the  normal distribution with mean $p$ and variance $p(1-p)/n(n-1)$, truncated at 0 and 1. Then  $\cD(\Gnp)$  and $\cB_{\hat p}(n)\mid_{\Sigma {\text{ is even}}}$ are a.q.e.\ provided that   $p(1-p)=\omega \log^3 n / n^2 $. 
\end{itemize}
\end{proposition}

 Note that the assumptions on $p$ and $m$ merely ensure that the graph, or its complement, has number of edges at least a small power of  $\log  n$; if this fails, the degree sequence is almost trivial and the graph  is likely to be uninteresting, either a set of independent edges or its complement.

Proposition~\ref{p:model} has significant implications. In particular, it was shown  in~\cite{degseq1} that  a result very similar to   Proposition~\ref{p:model}(ii) can be used, for whatever range of $p(n)$ it is valid for, to transfer general classes of properties to $\cD(\Gnp)$   from the independent binomial model ${\cal B}_p(n)$. We give details later in this section.  It was also observed in~\cite{degseq1}, using the known asymptotic formulae, that the Binomial Approximation Conjecture holds when $p=o(1/\sqrt n)$ or $p(1-p)> c/  \log n$.  However, the   rather large gap, where $p(1-p)$ is between roughly $1/\sqrt n$ and $1/\log n$, was still  open, as the appropriate enumeration results  were  lacking. This gap has prevented the new approach being fully utilised. Part (i) of the proposition is also appealing as a direct connection between independent binomials and $\cD(\Gnm)$, however   the event $\Sigma=2m$ is a fairly ``thin'' event so some results would require more finesse to be transferred.

In this article we 
introduce an approach to enumerating graphs by degree sequence that differs significantly from what has previously been applied to this or any similar problems. 
 This new method is versatile enough to be applied to enumeration problems for other discrete structures, as described in Section~\ref{s:final}. For this reason, our theoretical results are set in a framework slightly wider
 than is needed for the enumeration results derived in the present paper.

\subsection{Conjectures and results} 
In 1990 McKay and Wormald~\cite{MWhigh} unified the existing asymptotic formulae for the sparse and the dense case into one form and conjectured this form to hold also  for the gap in the range of degrees, as long as the degree sequences are close to regular. 
Throughout this paper  
 we use the following notation. Given a sequence $\dv=(d_1,\ldots, d_n)$, let $g(\dv)$ denote the number of graphs whose degree sequence is $\dv$, let $\mu= \mu  ( \dv)=d/(n-1)$  where $d=\frac{1}{n}\sum_{i=1}^n d_i$, and let $\gamma_2=(n-1)^{-2}\sum_{i=1}^n (d_i-d)^2$. 
We   refer to the following as the {\em Binomial Approximation Conjecture}.

\begin{conj} \cite[Conjecture 1]{MWhigh}  \thlab{conj1}  For some absolute constant $\eps>0$, 
 if $\dv=\dv(n)$ satisfies $\max_j | d_j-d| = o(n^\eps\min\{d, n-d-1\}^{1/2})$, $n\min\{d,n-d-1\}\to\infty$, and $\sum_i d_i$ is even, then  
 \bel{conjForm}
g(\dv) \sim \sqrt 2
 \exp\Big( {1\over 4} - {\gamma_2^2\over 4\mu^2(1-\mu)^2}  \Big)(\mu^ \mu(1-\mu)^{1-\mu}) ^{n(n-1)/2}  \prod {n-1\choose d_i}.  
\ee
  \end{conj}

\noindent{\bf Note.\ } 
The permitted domain of $\dv$  is compact for each $n$, so one can consider the `worst' sequence $\dv(n)$ for each $n$, and arrive at an equivalent way of stating the result: there is a function $\delta(n)\to 0$ such that the  relative error in ``$\sim$"   is bounded above in absolute value by   $\delta(n)$   for all   $\dv$ under consideration. This was the manner of stating Conjecture~\ref{conj2} below in~\cite{degseq1}.

Recall that $\cB_p(n)$ yields a random sequence consisting of $n$ independent  binomial variables $\text{Bin}(n-1,p)$. 
Let $\cB_m=\cB_m(n)$ denote $\cB_p(n)$ conditioned on the sum of the sequence being $2m$ and note that for a given sequence $\dv$ with $\sum_{i=1}^n d_i =2m$ we have that 
$$
\pr_{\cB_m}(\dv) = \binom{n(n-1)}{2m}^{-1}   \prod_{i=1}^n\binom{n-1}{d_i}
$$
which is, we recall, independent of $p$. 
Multiplying by $\size{\G(n,m)}$ and using Stirling's approximation 
shows that the formula 
\bel{formula}
\pr_{\cD(\Gnm)}  (\dv)  \sim \pr_{\cB_m}(\dv) 
\exp\Big( {1\over 4} - {\gtwo^2\over 4\mu^2(1-\mu)^2} \Big)  
\ee
is equivalent to the asymptotic formula in~\eqref{conjForm}, 
as long as $m$ and $n(n-1)-2m$ both tend to infinity.  

\thref{p:model} can be shown  to follow from a   
 conjecture made in 1997,  which we present below,  that is  weaker than \thref{conj1}. To state it,  we define the model $\Ep$ to be $\cB_p(n)$, conditioned on {\em even} sum, and then construct
$\Epp$  from $\Ep $ by weighting each $\dv$  with a weight depending 
only on $2m=\sum d_i$, such that the value $m$ has distribution ${\rm Bin}(n(n-1)/2,p)$. 

\medskip

\noindent{\bf Definition\ }A probability $p=p(n)$ is {\em acceptable} if $p(1-p)n^2= \omega \log n$ and 
there is a set-valued function $R_p(n)$ of integer sequences of length $n$ with even sum, such that   both of the following hold. 
 \begin{enumerate}[(i)]
\item For   $\dv\in R_p(n)$  
\bel{formula1}
\pr_{\cD(\Gnp)} (\dv) \sim \pr_{\Epp}(\dv) 
\exp\Big( {1\over 4} - {\gtwo^2\over 4\mu^2(1-\mu)^2} \Big).
\ee
 \item In each of the models $\Ep$ and $\cD(\Gnp)$, we have
$\pr\big(R_p(n)\big) = 1 - n^{ -\omega }$.
 \end{enumerate}

The conjecture in~\cite{degseq1} is as follows.
\begin{conj} [McKay and Wormald \cite{degseq1}]  \thlab{conj2}  If  $p(1-p)n^2= \omega \log n$ then $p(n)$ is acceptable.  \end{conj}
  
The fact that  the exponential factor in~\eqn{formula1} is sharply concentrated near 1 was used to prove~\cite[Theorem 2.6(b)]{degseq1} (see Section~3 of that paper), which we do not state here. The first part of~\cite[Theorem 2.6(b)]{degseq1} gives bounds on the differences between the expected values of random variables in $\cD(\Gnp)$ and the model $\cB_{\hat p}(n)\mid_{\Sigma {\text{ is even}}}$ 
of \thref{p:model}(ii)  for any acceptable $p$. The second part of that theorem similarly bounds   the difference between expectations in
$\cD(\Gnm)$  and $\cB_m(n)$, for any $m$ such that $2m/n(n-1)$ is acceptable. In particular,   \thref{conj2}   implies Proposition~\ref{p:model} via~\cite[Theorem 2.6(b)]{degseq1} with $X_n$ defined as the indicator of the event $H_n$ that appears in the definition of a.q.e.  We omit further details of this, since for applications of our main result, one can look at all the consequences of Conjecture~\ref{conj2} in~\cite{degseq1}, of which  Proposition~\ref{p:model} is just an example. In particular, further results in~\cite[Section~4]{degseq1} show that for  many properties in $\Gnp$,  the conditioning on parity in Proposition~\ref{p:model}(ii) has negligible effect, so one can consider purely independent binomials. Sometimes the  integration  required to deal with the distribution of  $\hat p$ also has negligible effect, and in any case,~\cite[Theorems 3.7 and 3.8]{degseq1} give some general results on carrying out the integration. To our knowledge, there are so far no detailed applications of the consequence for $\Gnm$ given in \thref{p:model}(i).

It was shown in~\cite[Theorem~2.5]{degseq1}  
that $p$ is acceptable if either $\omega \log n /n^2 \le p(1-p) = o(n^{-1/2})$ or $p(1-p)\ge c/\log n$ for some $c>2/3$, using  the known enumeration results in the respective range. The essence of the proof of~\cite[Theorem~2.5]{degseq1}  shows that if~\eqn{conjForm} of Conjecture~\ref{conj1} holds for all $d$ near some $d_0$, then by known concentration results,  $p=d_0/(n-1)$ is acceptable. The same considerations prove that Conjecture~\ref{conj2} follows from Conjecture~\ref{conj1}.

In this paper, we prove 
Conjecture~\ref{conj1} and hence Conjecture~\ref{conj2} and \thref{p:model} in the remaining gap range. 
Our main theorem is the following.
\begin{thm}\thlab{t:graphwide}
Let $\mu_0>0$ be a sufficiently small constant, and let $1/2 < \alpha<3/5$. 
Let $n$ and $m$ be integers and set $d=2m/n$ and 
assume that $\mu = d/(n-1)$ satisfies $ \mu \leq \mu_0$ and, for all fixed $K>0$, $(\log n)^K/n =o(\mu)$. 
Let $\D$ be the set of sequences $\dv$ of length $n$ satisfying 
$\sum_i d_i = 2m$ and $\size{d_i - d} \le d^{\alpha}$ for all $i\in[n]$. 
Then 
uniformly for all $\dv\in\D$ we have 
$$
\Pr_{\cD(\G(n,m))}(\dv) 
	= \Pr_{\cB_m}(\dv) \exp\Big( {1\over 4} - {\gtwo^2\over 4\mu^2(1-\mu)^2} \Big) 
	\left(1+ 
	O\left( \frac{\log^2n}{\sqrt{n}}+
	d^{5\alpha-3} 
	 \right)
	 \right).
$$
\end{thm}

\begin{cor}\thlab{c:main}
Conjectures~\ref{conj1} and~\ref{conj2}  are both true.
\end{cor}

In Section~\ref{s:completinggraphs} we provide a hand-checkable proof of
\thref{t:graphwide} under the restriction that
$\max\{\size{d_i - d}\}=O(\sqrt{d\log n})$ and
$\sum_i (d_i - d)^2 \leq 2dn$. (See \thref{t:handCalc}.)
As a corollary of this and known results, we obtain Conjecture~\ref{conj2}. 
Extending the expansions of functions involved to more terms using computer algebra lets us  obtain a proof of \thref{conj1}.

Our method is markedly different from those previously used for this problem, the most successful of which used either the ``configuration model'' with ``switchings'', or the estimation of integrals representing coefficients of appropriate generating functions. Instead, we 
derive a set of three equations relating the numbers of graphs with almost the same degree sequence, and the probability that they contain a given edge or a given pair of incident edges. The equations are derived by examining the operation of moving one end of an edge from one vertex to another, and the operation of deleting an edge. These equations are analysed from the perspective of fixed points of a contraction mapping. 

Note that if we build up a random graph by inserting $m$ edges using the method of choosing each endpoint of each edge independently at random, the resulting multigraph has   exactly the multinomial degree distribution, that is,  $\cB_m(n)$ (see for example Cain and Wormald~\cite{encore}). The only difficulty is that the result can contain  loops and multiple edges. For Proposition~\ref{p:model}, one would need  to show that conditioning on the absence of  loops and multiple edges does not significantly affect the degree distribution. This is difficult since the probability of no loop or multiple edges occurring is extremely small even for moderate densities. We believe that our approach to this enumeration problem gives an intuitive explanation  of why the 
distribution of degrees in a random graph is so close to the binomial model, since it shows directly that the distribution of graph counts behaves ``locally" in the appropriate way.   
Such an explanation is absent from the main methods used previously, in~\cite{MWlow} and~\cite{MWhigh}.

As a by-product of our proof, we obtain asymptotic formulae for the edge probabilities in a random graph with a given degree sequence. For a sequence $\dv$ of length $n$,
 let $\G(\dv)$ be a graph chosen uniformly at random from all graphs that have degree sequence $\dv$, and let $\Pr_{ab}(\dv)$ be the probability that the edge $v_av_b$ is present in $\G(\dv)$.

\begin{thm}
	Let  $\D$ be as in \thref{t:graphwide} and let $a,b\in[n]$, $a\neq b$. 
	Then for all $\dv\in \D$ 
	\begin{align*}
	\Pr_{ab}(\dv) &= \frac{d_ad_b}{d(n-1)}
	\left( 1- \frac{(d_a-d)(d_b-d)}{d(n-1-d)}\right) 
	+O\left(  \frac{\sqrt{d \log n}}{n^2}+\frac{(\sqrt{d \log n})^3}{n^3}\right).
	\end{align*}
\end{thm}
We also give a more accurate but more complicated formula for $\Pr_{ab}(\dv) $ in Lemma~\ref{l:mapleHigher}(ii), and an approximation for sparser degree sequences in~\eqn{Pavsparse}. Our results also provide formulae for the probability that a given path of length 2 occurs in the same graphs.

In Section~\ref{s:prelim} we provide the notation we use in this paper and some preliminary results. 
In Section~\ref{s:recursive} we derive certain recursive formulae that are the core of our method. 
In Section~\ref{s:sparseGraphs} we reprove enumeration results for the sparse case to illustrate how the recursive formulae from Section 3 are to be used to obtain explicit formulae. We also take this opportunity to provide a general template on how our proofs in the later parts are structured. 
In Section~\ref{sec:FunOps} the recursive formulae are turned into operators. Section~\ref{s:completinggraphs} contains the proof of \thref{t:graphwide} under stricter assumptions on the degree spreads, and gives the proof of \thref{conj2}, all using easily hand-checkable calculations. As explained above, this is enough to deduce our first major goal, \thref{p:model}. In Section~\ref{s:wider} we provide the full proof for \thref{t:graphwide}. We finish the paper with some concluding remarks in Section~\ref{s:final}, and the Appendix gives details of some routine calculations.

\subsection{Further comments} \lab{s:further}
 
 Studying the degree sequence of a random graph is not the only significant potential use of enumeration formulae for graphs with given degrees. A large number of properties of random regular graphs have been shown using them, in particular, results on subgraphs of random regular graphs or random graphs with given degrees. (McKay~\cite{McKay_ICM} refers to many  examples.) The asymptotic distribution of the number of edges in the giant component of the random graph, and the size of its 2-core, were first obtained by Pittel and Wormald~\cite{pw2005} with heavy use of these formulae.  
Additionally, Kim and Vu~\cite[Equation (3.1)]{KV} heavily used the asymptotic formula for the number of $d$-regular graphs, in the known range, to give strong relations between random graphs and random regular graphs. The lack of a formula in the denser case restricted their work to the sparse case. The results in the present paper could have helped here, and would presumably have simplified the extension of Kim and Vu's work to denser cases, given recently by  Dudek, Frieze, Ruci{\'n}ski and {\v S}ileikis~\cite{dfrs2017}.

Various results on random $d$-regular graphs \cite{bky2017,ksvw2001} use the configuration model and a sophisticated analysis of switchings, which are methods developed for asymptotic enumeration of graphs by degree sequence when the degrees are relatively small. The methods  from this paper may well have similar applications, even in cases where the enumeration formulae themselves do not apply.
 
We note that Isaev and McKay~\cite{IM} have given a major further development  of the methods in~\cite{MWhigh} and~\cite{BH} to obtain further results useful in the case of very dense graphs. They have  also (unpublished) announced progress in pushing this method towards sparser cases.  Also after our project was   under way, Burstein and Rubin~\cite{br2015}   presented an idea of comparing numbers of graphs with given degree sequences that is somewhat related to the techniques in the present paper, but differs in several significant ways. Their approach would give a formula that is valid up to maximum degree $n^{1-\delta}$ for any fixed $\delta>0$, using a finite amount of computation.    However, their general result is not explicit enough  to enable the derivation of results as simple as the formula~\eqn{conjForm}. In particular, we are not aware of any claims of extending the range of validity of the Binomial Approximation Conjecture, apart from the present paper.


\section{Preliminaries and Notation.}
\lab{s:prelim}

For the reader's convenience, we solidify some notation here. Our {\em graphs} are simple, that is, they have no loops or multiple edges. We write    $a\sim b$ to mean that $a/b\to 1$,    $f=O(g)$ if $|f|\le Cg$ for some constant $C$,  and $f=o(g)$ if $f/g\to 0$.  We use $\omega$ to mean a function going to infinity, possibly different in all instances. 
We use $1\pm \xi$ to denote a quantity in the interval $[1-\xi,1+\xi]$.
Also   $\binom{[n]}{2}$   denotes the set of $2$-subsets of the set $[n]=\{1,\ldots , n\}$. 
We often consider a vector $\dv = (d_1,\ldots, d_n)$, and use $\Delta$ or $\Delta(\dv)$ to denote $\max_i d_i$, in line with the notation for maximum degree of a graph. 
Recall also from the introduction that $g(\dv)$ denotes the number of graphs whose degree sequence is $\dv$,  $\mu= \mu ( \dv)=d/(n-1)$  where $d=d(\dv) = \frac{1}{n}\sum_{i=1}^n d_i$, and  $\gamma_2=(n-1)^{-2}\sum_{i=1}^n (d_i-d)^2$. 
We write $M_1=M_1(\dv)=\sum_{i=1}^n d_i$. If $\dv$ is the 
 degree sequence of a graph, this is the total degree, or twice the number of 
edges.
Parity is an important issue, so we say a vector $\dv$ is {\em even} if $M_1(\dv)$ is even, and {\em odd} otherwise.
Finally,  in this paper  multiplication by juxtaposition has precedence over ``$/$", so for example $j/\mu n^2= j/(\mu n^2)$.

We first  state a simple result by which we leverage an enumeration result from comparisons of related numbers.  
 
 \begin{lemma}\thlab{l:lemmaX} 
Let  $\cS$ and $\cS'$ be  probability spaces with the same underlying set $\Omega$.  
Let $G$ be a graph with vertex set  $\W\subseteq \Omega$  such that $ \pr_\cS(v),\pr_{\cS'}(v) >0$  for all $v\in \W$. Suppose that  $\eps_0,  \delta >0$    such that  
 $\min\{ \pr_\cS(\W), \pr_{\cS'}(\W)\}>1-\eps_0>1/2$,  and such that  for every edge $uv$ of $G$,  
$$
\frac{\pr_{\cS'}(u)}{\pr_{\cS'}(v)}=   e^{O( \delta )}  \frac{\pr_{\cS}(u)}{\pr_{\cS}(v)}   
$$
where the constant implicit in $O(\cdot)$ is  absolute. 
 Let $r$ be an upper bound on the diameter of $G$ and assume $r<\infty$. Then for each $v\in \W$ we have
 $$
  \pr_{\cS'}(v)  =e^{O(r \delta +\eps_0)}  \pr_\cS(v) ,
$$
with again a   bound uniform for all $v$.
\end{lemma}
 
\begin{proof} 
For any $u$, $v\in \W$ we may take  a telescoping product of ratios along a path joining $u$ to $v$ of length at most $r$. This gives 
$\pr_{\cS'}( u )/\pr_{\cS'}(v)=e^{O(r \delta )} \pr_{\cS}(u)/\pr_\cS(v)$. Summing over all $u\in \W$, gives 
$$
\frac{\pr_{\cS'}(\W)}{\pr_{\cS'}(v)}= \frac{e^{O(r \delta )}\pr_{\cS}(\W)}{ \pr_\cS(v)},
$$
and the claim follows using the lower bound $1-\eps_0$ on $\pr_{\cS}(\W)$ and  $\pr_{\cS'}(\W) $. 
 \end{proof}

Using the lemma calls for evaluating the ratios of probabilities in an ``ideal'' probability space, the one by which we are approximating the ``true'' space. This leads to computing ratios in the conjectured formulae. One we need several times is the following. Let $H(\dv)$ denote the conjectured formula in the right hand side of~\eqn{formula}, apart from the error term, that is
$$
H(\dv)=\pr_{\cB_m}(\dv) 
\exp\Big( {1\over 4} - {\gtwo^2\over 4\mu^2(1-\mu)^2} \Big).
$$
We use $\ea$ to denote the elementary unit vector with 1 in its $a^{\text{th}}$ coordinate. 
Noting that $\mu= d/ (n-1)$  where $d=M_1/n$  (a function of the degree sequence) has the same value  for  both cases $\dv-\ea$ and $\dv-\eb$,
 and recalling  $\gamma_2=\gamma_2(\dv)=  (n-1)^{-2}\sum_i (d_i-d)^2$,  we have (for {\em all}  odd  $\dv$ with $\Delta =\max d_i \le n/2 $  and $\mu>0$) 
\begin{align}
\frac{H(\dv-\ea)}{H(\dv-\eb)}&=\frac{d_a(n -d_b)}{d_b(n -d_a)}\exp\left(\frac{ (d_a-d_b)\big( \gamma_2+O(\Delta/n^2)\big)   }{(n-1)^2 (\mu')^2(1-\mu')^2 } \right)\nonumber\\
&=   \frac{d_a(n- d_b)}{d_b(n -d_a)}\exp\left(\frac{ (d_a-d_b)   \gamma_2  }{ d^2 (1-\mu')^2 } +O\big( \Delta^2/(dn)^2\big) \right). \lab{conjRatio}
\end{align}
 where $\mu'=\mu(\dv-\ea)=\mu(\dv-\eb)$, and of course $\gamma_2$, $d$, and $\Delta$ are defined with respect to $\dv$.

Next we turn to some  issues involving existence of graphs with a given degree sequence.
For counting purposes we will be considering graphs with vertex set     $V=[n]$.  Let $\cA=\cA(n)\se \binom{ [n] }{2}$ be a set which we call {\em allowable pairs}.  Note that as usual we regard the edge joining vertices $u$ and $v$ as the unordered pair $\{u,v\}$, and denote this edge by $uv$ following standard graph theoretic notation. 
A sequence $\dv:= (d_{1},\ldots,d_{n})$ is called  
{\em $\cA$-realisable}
if there is a graph $G$ 
on vertex set $V$ such that vertex $a\in V$ has degree $d_a$ and all edges of $G$ 
are allowable pairs. 
In this case, we say $G$ {\em realises $\dv$ over $\cA$}.  In standard terminology, if $\dv$ is  $\binom{V}{2}$-realisable, it is  {\em graphical}. Let $\cG_{\cA}(\dv)$ be the set of all graphs that realise $\dv$ over $\cA$. 
In this paper, we are particularly interested in the case that $\cA=\cA^{\mathrm{gr}} =\binom{V}{2}$. Then $\cG_\cA(\dv)$ is the set of all 
graphs $G$ on vertex set $[n]$ 
 that have degree sequence $\dv$. 
By using different definitions of $\cA$, it is also possible to model other enumeration problems. 
Throughout this paper, all graphs are finite and simple (i.e.~have no  loops or multiple edges).

Let $E\se \cA$, i.e.~a subset of the allowable edges. 
We write $\Num_E(\dv)$ and $\Num^*_E(\dv)$ 
for the number of graphs $G\in \cG_{\cA}(\dv)$ 
that contain, or do {\em not} contain, the edge set $E$, respectively.
We abbreviate $\Num_E(\dv)$ to $\Num_{ab}(\dv)$   if  $E={\{ab\}}$ (i.e.~contains the single edge $ab$), 
and put $\Num(\dv)=\size{\cG_{\cA}(\dv)}$. (When  $\Num$ and similar notation is used, the set $\cA$ should be clear by context.)

We pause for a notational comment. In this paper, a subscript  $ab$  is always  interpreted as an ordered pair $(a,b)$ rather than an edge (and similar for triples). This is irrelevant for $\Num_{ab}(\dv) =  \Num_{ba}(\dv)$ since the two ordered pairs signify the same edge, but the distinction is important with other notation. 

Let 
$$P_E(\dv)  =\frac{\Num_{E}(\dv)}{\Num(\dv)},$$ 
which is the probability 
that the edges in $E$ are present in a graph $G$ that 
is drawn uniformly at random from $\cG_{\cA}(\dv)$. 
Of particular interest are the probability of a single edge $av$ and a path $avb$, 
for which we simplify the notation to
$$
P_{av}(\dv) =P_{\{av\}}(\dv),\qquad \YPavb_{avb}(\dv) =P_{\{av, bv\}}(\dv) 
$$
where $a$, $b$ and $v$ are all distinct.
We will use the following trick several times to switch between degree sequences of differing total degree.      
\begin{lemma}\thlab{trick17}
Let $a v\in \cA$ and let $\dv$ be a sequence of length $N$. 
Then 
\begin{align*}
\Num_{av}(\dv) 
	&= \Num (\dv - \ea - \ev) -  \Num_{av} ({\bf d} - \ea - \ev)\\  
	&= \begin{cases} 
		\Num ({\bf d} - \ea - \ev) (1- P_{av}(\dv- \ea - \ev)) & \mbox{if } \Num({\bf d} - \ea - \ev) \ne 0\\ 
		0 & \mbox{otherwise.}
	\end{cases}
\end{align*}
\end{lemma}
\begin{proof}
Removing an edge $av$ from a graph in $\cG_{\cA}(\dv)$ shows that the number of graphs with that edge is the same as the number of graphs in $\cG_{\cA}(\dv - \ea - \ev)$ and {\em no}  edge between $a$ and $v$. (The general form does not apply when $\Num ({\bf d} - \ea - \ev)=0$ only because $P_{av}(\dv- \ea - \ev)$ is then technically undefined.)
\end{proof}
For vertices $a,b \in V$, if $\dv$ is a sequence such that $\dv-\eb$ is $\cA$-realisable,   
we define  
$$
 R_{ab}(\dv)  =\frac{\Num(\dv-\ea)}{\Num(\dv-\eb)}.
$$ 
 With $\cA$ understood we write $\mu =\mu(\dv)= \frac12 M_1/  \size{\cA}$, which is the relative edge density of an $\cA$-realisable graph with degree sequence $\dv$. 
When $\cA$ is undefined, as in the introduction, then the default assumption is that $\cA=  \binom{ [n] }{2}$  (so that this definition of $\mu$ extends the one given in the introduction). 
Finally, for a vertex $a\in V$, we set $\cA(a) =\{v\in V: av \in \cA\}$, the ``projection'' of $\cA$ onto  
the edges incident with vertex $a$, and, with $\dv$ understood, we use $\cA^*(a)$ for the set of $v\in \cA(a)$ such that $\Num_{av}(\dv)>0$.

We can bound the probability of an edge in a simple way using the following switching argument.\begin{lemma}\thlab{l:simpleSwitching}
Let $\dv$ be a  graphical  sequence of length~$n$ with $\sum d_i=dn$. 
Then for any $a$ and $v$ in  $[n]$ we have 
$$
 P_{av}(\dv)\le \frac{\Delta^2}{dn \left(1-  \Delta(\Delta+2)/dn\right)}. 
$$
\end{lemma}
\begin{proof} 
For each graph $G$ with degree sequence $\dv$ and an edge joining $a$ and $v$, we can perform a switching (of the type used previously in graphical enumeration) by removing both $av$ and  another randomly chosen edge $a'v'$, and inserting the edges  $aa'$ and $vv'$, provided that no loops or multiple edges are so formed. The number of such switchings that can be applied to $G$ with the vertices of each edge ordered, is at least 
$$
dn-2(\Delta+1)\Delta 
$$ 
since there are   $dn$  ways to choose $a'$ and $v'$ as the ordered ends of any edge,  whereas the number of such choices that are ineligible is at most the number of choices with $a'$ being $a$ or a neighbour of $a$ (which automatically rules out  $a'=v$), or similarly for $v'$. On the other hand, for each graph $G'$ in which  $av$ is {\em not} an edge, the number of ways that it is created by performing such a switching backwards is at most $\Delta^2$. Counting the set of all possible switchings over all such graphs $G$ and $G'$ two different ways shows that the ratio of the number of graphs with  $av$  to the number without  $av$ is at most
$$
\beta:=\frac{\Delta^2}{dn  -2\Delta(\Delta+1) }.
$$
Hence $P_{av}(\dv)\le \beta/(1+\beta)$, and the lemma follows in both cases. 
\end{proof}


We finish the section with  some simple sufficient conditions for a sequence to be graphical.  We need this since our degree switching arguments require that $\Num(\dv')>0$ for several sequences $\dv'$ that are 
related to a {\em root} sequence $\dv$. 
 \begin{lemma}\thlab{l:graphical}
For $\eps>0$, the following holds for $n$ sufficiently large. 
Let  $ d_i\ge 0$   be integers for all $1\le i \le  n$, with  $ \sum_{1\le i\le n } d_i$ even.    Then  there exists a   graph with degrees $d_1,\ldots, d_n$  provided that either of the following hold. 
\begin{enumerate}[(a)]
\item \label{l:graphicali}  
There exists $0<\mu<1-  \eps$ with  $\max\size{\mu - d_i/n } <\eps\mu$.
\item \label{l:graphicalii} We have 
$1\le d_i\le 2\sqrt n - 2$ for   $1\le i \le n$.
\end{enumerate}
 \end{lemma}
\begin{proof}
Koren~\cite[Section 1]{K} observed that the classical Erd{\H o}s-Gallai conditions for the existence of a graph with degrees $d_1,\ldots, d_n$ are equivalent to the following: $\sum d_i$ is even, and whenever $S\cap T= \emptyset \ne S\cup T \subseteq [n]$ we have
$$
\sum_{i\in S} d_i-\sum_{j\in T}d_j\le s(n-1-t)
$$
where $s=|S|$ and $t=|T|$. 

For~\eqref{l:graphicali}, it suffices to have $ s(n-1-t)-s\mu n(1+\eps)+t\mu n(1-\eps)\ge 0 $. We only need to check the extreme values of $t$, i.e.~$t=0$ and $t=n-s$. In both cases, the function is easily seen to be non-negative for all appropriate $s$ and $n$ sufficiently large, noting that $\mu\le 1-\eps$. 

For~\eqref{l:graphicalii},  we only need $s\Delta-t\le s(n-1-t)$ (recall  that $1\le d_i \le \Delta$). Again taking the extreme values of $t$, only the larger one gives a restriction, being $s(\Delta+2-s)\le n$, which is satisfied because $\Delta\le 2\sqrt n -2$.
\end{proof}


\section{Recursive relations}\lab{s:recursive}
In this section, we derive certain recursive formulae 
for the probability and ratio functions $P_{av}(\dv)$ 
and  $R_{ab}(\dv)$. These identities will serve as a motivation for operators that we define in the 
next section. With a view to further applications of this  work elsewhere, we consider  an arbitrary set $\cA$ of allowable pairs.

Our first result expresses the edge probability $P_{av}$, the ratio $R_{ab}$, and the path probability $Y_{ivj}$  in terms of each other.

\begin{proposition}\thlab{l:recurse}
Let $\dv$ be a sequence of length  $n$ and let $\cA\se \binom{[n]}{2}$.
\begin{itemize}
\item[$(a)$] Let $a,v\in V$. If $\Num_{av}(\dv)>0$ 
then 
	\begin{equation*}
	P_{av}(\dv) =d_{v} \Bigg(\sum_{b\in \cA^*(v)}
	 R_{ba} ( \dv- \ev)
	\frac{1-P_{bv}(\dv - \eb - \ev)}
	{1-P_{av}(\dv - \ea - \ev)}\Bigg)^{-1}.
	\end{equation*}
\item[$(b)$] Let $a,b\in V$. If $\dv-\eb$ is $\cA$-realisable then 
	\begin{align} \lab{Ratformula}
	 R_{ab} ( \dv) &= \frac{d_{a}}{d_{b}}\cdot \frac{ 1-\Bad(a,b, {\bf d} - \eb)}{ 1-\Bad(b,a, {\bf d} - \ea)}, 
	\end{align}
	where 
	\begin{align}\lab{eq:bad}
	\Bad(i,j, \dw) 
	& = 
	 \frac{1}{d_{i}}\Bigg(\sum_{ v \in \cA(i)\setminus \cA(j)}  P_{iv}(\dw)   +
	 \sum_{ v \in \cA(i)\cap \cA(j) }  Y_{ivj}(\dw) \Bigg), 
	\end{align}
	provided that  $\Bad(b,a, {\bf d} - \ea) \ne 1$.
 
\item[$(c)$] Let $a,v,b$ be distinct elements of $V$. If  $\dv-\ea-\ev$ is $\cA$-realisable then 
\begin{equation*}
	\YPavb_{avb}(\dv) = 
	\frac{P_{av}(\dv)\big( P_{bv}(\dv - \ea - \ev)- \YPavb_{avb}(\dv - \ea - \ev)\big)}
	{1-P_{av}(\dv - \ea - \ev)} .
	\end{equation*}
\end{itemize}

\end{proposition}

\noindent {\bf Remarks\ }
\begin{enumerate}
\item In (a), the summation over all $b$ in $\cA^*(v)$, instead of $\cA(v)$, is merely for the technicality that 
$P_{bv}(\dv-\eb-\ev)$ is otherwise undefined.

\item The condition $\Num_{av}>0$ in part (a) does not reduce the practical usefulness of the lemma since the degree sequences where this condition fails for allowable edges $av$ with $d_a,d_v>0$ are pathological enough that our method fails on those for  other reasons.

\item For  the range of the summations in $B(i,j,\dv')$, when $\cA=\binom{[n]}{2}$ (as in the applications in this paper) we have
	$\cA(i)\sm \cA(j)=\{j\}$ and $\cA(i)\cap \cA(j)=[n]\sm\{ i,j\}$. 
\end{enumerate}

\begin{proof}
To prove part (a) of the lemma, let $\dv$ be an $\cA$-realisable  sequence. 
Then every graph $G\in \cG_\cA(\dv)$ contributes exactly $d_v$ to 
$ \sum_{ b\in \cA^*(v) } \Num_{bv}(\dv).$
Hence, since $\Num_{av}(\dv)>0$ we may write  
\begin{align*}
d_v &= \sum_{\substack{b\in \cA^*(v) }} \frac{\Num_{bv}(\dv)}{\Num(\dv)}
	= P_{av}(\dv) \sum_{\substack{b\in \cA^*(v) }}  \frac{\Num_{bv}(\dv)}{\Num_{av}(\dv)}\\
	&= P_{av}(\dv) \sum_{\substack{b\in \cA^*(v) }} 
		\frac{\Num({\bf d} -\eb - \ev) (1- P_{bv}({\bf d} - \eb-\ev ))}
		{\Num({\bf d} - \ea-\ev) (1- P_{av}({\bf d}- \ea-\ev))},
\end{align*}
by Observation~\ref{trick17}, noting that 
$\Num({\bf d} -\eb - \ev)\ge\Num_{bv}(\dv)>0$ by defnition of $\cA^*(v)$. 
 Part (a) follows since by definition   
$$
\frac {\Num({\bf d} - \eb-\ev )}{\Num({\bf d} - \ea-\ev)}
=  R_{ba} ( \dv - \ev),
$$
and also noting that the summation is non-zero because $a\in \cA^*(v)$.

To prove part (b) of the lemma, assume that $\dv-\eb$ is $\cA$-realisable. First note that if $a=b$ then by definition $R_{ab} ( \dv)=1$ and so the formula is correct. We can therefore assume henceforth that $a\ne b$.
Let $J_1$ be the set of graphs in $\cG_\cA(\dv-\eb)$ 
with a distinguished edge incident to vertex $a$, and $J_2$ the set of graphs in $ \cG_\cA (\dv-\ea)$ with a distinguished edge incident to $b$. Then $|J_1|= d_a  {\Num(\dv-\eb)}$ and $|J_2|= d_b  {\Num(\dv-\ea)}$.  
Applying a {\em degree switching} to  $G\in J_1$  consists of deleting the distinguished edge $av$ and adding a new distinguished edge $bv$, to produce  a graph $G'\in J_2$. The degree switching cannot be performed, i.e.\ is not valid, if  $bv\notin \cA$, which includes the case that $v=b$,  or  $bv$ is   an edge of $G$.  
Now let $G\in \cG_\cA(\dv-\eb)$  be picked  uniformly at random 
and let $v$ be a random neighbour of $a$ in $G$.  Let $E$ be the event $E_\cA \cup E_D$ where $E_\cA$ is the event that $bv\not\in \cA$ and $E_D$ is the event that $bv$ is an edge of $G$, and define  $ \Bad(a,b, \dv-\eb) =\pr(E)$ (which we show further below to satisfy~\eqref{eq:bad}). 
Then the number of valid degree switchings  
is 
$$ d_a{\Num(\dv-{\bf e_b})} (1-\Bad(a, b, \dv-\eb)).$$ 
We may count  the same switchings from the other direction, i.e.\ starting with an element of $J_2$, using the same argument, and in this case $\Num ({\bf d} - {\ea})=0$ is permissible.  Equating the two counts gives
$$
\frac{\Num ({\bf d} - {\ea})}{\Num({\bf d}-{\bf \eb}) } = \frac{d_a}{d_b}\cdot \frac{ 1-\Bad(a, b, {\bf d}-{\eb})}{ 1-\Bad(b, a, {\bf d} -\ea)},
$$
where the denominator is non-zero by the hypotheses of (b), noting that $d_b>0$ because $\dv-\eb$ is $\cA$-realisable. This gives~\eqn{Ratformula}.

The event $E_D$ can only happen if the vertex $v$ is a neighbour both of $a$ 
and of $b$, and hence (as we still assume $a\ne b$)
$$
\pr(E_D) =\frac{1}{d_a}\sum_{v\in \cA(a)\cap \cA(b)} \YPavb_{avb}(\dv-\eb).
$$
On the other hand, the vertex $v$ is always a neighbour of $a$ in a graph $G\in \cG_\cA(\dv)$, and thus 
$$
\Pr(E_\cA) = \Pr(\{ v\notin \cA(b)\}) = \Pr(\{ v\in \cA(a)\sm \cA(b)\})=\frac{1}{d_a} \sum_{ v \in \cA(a)\setminus \cA(b)  } P_{av}(\dv-\eb) .
$$
Noting that   
$ E_\cA\cap E_D=\emptyset$, so   $\Pr(E)=\Pr(E_\cA)+\Pr(E_D)$, we obtain  the stated formula for $\Bad(a,b,\dv-\eb)$. For  $\Bad(b,a,\dv-\ea)$ we can just swap $a$ and $b$.

For (c),   
the assumptions imply that $\Num_{av}(\dv) > 0$, and hence by definition, 
\begin{align}\lab{eq:aux792}
\YPavb_{avb}(\dv) &= \frac{\Num_{\{av, bv\}}(\dv)}{\Num(\dv)}
	= P_{av}(\dv) \cdot \frac{\Num_{\{av, bv\}}(\dv)}{\Num_{av}(\dv)}. 
\end{align}
 Now \thref{trick17} says that
\begin{align*}
\Num_{ av }(\dv)  
&=\Num (\dv - \ea-\ev) -  \Num_{ av }(\dv -  \ea - \ev)\\
 & =   \Num (\dv - \ea-\ev)\big(1-  P_{av}(\dv -  \ea - \ev)\big)
\end{align*}	
and similarly (c.f.\ the proof of \thref{trick17}) we have
\begin{align*}
\Num_{\{av, bv\}}(\dv)
	&=\Num_{bv}(\dv - \ea-\ev) -  \Num_{\{av, bv\}}(\dv -  \ea - \ev) \\ 
	&= \Num (\dv - \ea-\ev) \big(  P_{bv}(\dv - \ea - \ev)- P_{abv}(\dv - \ea - \ev)  \big).
\end{align*} 
Part (c)   follows upon substituting these expressions into~\eqn{eq:aux792}.
\end{proof}


\section{The general method and the sparse case  for graphs}
\lab{s:sparseGraphs}

In this section we  present a simple application of the recursive relations found in Section~\ref{s:recursive}. This is  a completely new derivation of a known formula for the number of sparse graphs with a given  degree sequence. We first give a template of the method, since it is used again in the other proofs in this paper and can be used elsewhere.

\bigskip
\noindent {\bf Template of the method}
\smallskip

Step~1.~Obtain an estimate of the ratio  between the numbers of graphs of related degree sequences, using Proposition~\ref{l:recurse}. This step is the crux of the whole argument. 

Step 2. By making suitable definitions, we cause this ratio to appear as the expression \linebreak $\pr_{\cS'}(u)/ \pr_{\cS'}(v)$ for some probability space $\cS'$ on an underlying set $\Omega$ in an application of  Lemma~\ref{l:lemmaX}. Thus, $\Omega$ is the set of degree sequences, with probabilities in $\cS'$ determined by  the random graph under consideration, and the graph $G$ in the lemma has a suitable vertex set $\W$ of such sequences. Each edge of $G$ is in general a  pair of degree sequences   $\dv-\ea$ and $\dv-\eb$ of the form occurring in the definition of $R_{ab}( \dv) $. Having defined $G$, we may call any two such degree sequences {\em adjacent}.  

Step 3. Another probability space $\cS$ is defined on $\Omega$, by taking a probability space $\cB $ directly from a joint binomial distribution,   together with a function $\widetilde H(\dv)$ that varies quite slowly, and defining probabilities in $\cS$ by the equation
$ \pr_\cS(\dv)=    \pr_{\cB }(\dv) \widetilde H(\dv) /\ex_{\cB} \widetilde H$.  

Step 4.
Using sharp concentration results, show that  $P(\W) \approx 1$ in both of the  probability spaces $\cS$ and $\cS'$ (where, by $\approx$, we mean approximately equal to, with some specific error bound in each case). As part of this, we show that $\ex_{\cB } \widetilde H \approx 1 $. At this point, we may specify $\eps_0$ for the application of Lemma~\ref{l:lemmaX}.
 
Step 5.  Apply Lemma~\ref{l:lemmaX} and the conclusions of the previous steps to deduce   $P_{\cS'}(\dv) \approx P_\cS(\dv)\approx \pr_{\cB }(\dv) \widetilde H(\dv)$. Upon estimating the errors in the approximations, which includes bounding the diameter of the graph $G$, we obtain an estimate for the probability $P_{\cS'}(\dv)$ of the random graph having degree sequence $\dv$ in terms of a known quantity.
\medskip

Recall that  given a sequence $\dv$ we write $\Delta(\dv)=\max_i d_i$,  $M_1(\dv)=\sum_i d_i$, $d=d(\dv)=M_1/n$, and $\mu=\mu(\dv)=d/(n-1)$. Note that in the following result the condition $\Delta(\dv)^6+n^{ \eps} = o(nd^2)$ implies in particular that $m/\sqrt{n} \to\infty$. This restriction is imposed just for simplicity; the technique can still apply in the (less interesting) extremely sparse case. Also recall the probability spaces of random sequences $\G(n,m)$ and $\cB_m(n)$ from Section~\ref{s:intro}.

 \begin{thm}\thlab{t:sparseCase}
Let $\eps>0$,  let $n$  and $m$ be integers.  
Let  $\D$ be  any set of sequences of length $n$ with
 $\sum_i d_i=2m $ for all $\dv\in\D$ such that 
 $\Delta(\dv)^6+n^{ \eps} = o(nd^2)$ uniformly for all $\dv\in\D$. 
Then uniformly for $\dv^*\in\D$ we have
$$ 
\pr_{\G(n,m)}(\dv^* )= \pr_{\cB_m}(\dv^*)\exp\left( \frac{1}{4}-\frac{\gtwo^2}{4\mu^2(1-\mu)^2}\right)\bigg(1+O\bigg( \frac{ \Delta(\dv^*)^6+n^{ \eps}}{nd^2 } +n^{\eps-1/2}  \bigg)  \bigg). 
$$
\end{thm}
\begin{proof}
We can clearly assume that $\D$ is nonempty and we fix  $\dv^*\in\D$. We first  estimate ratios of probabilities for adjacent sequences that are typical in the binomial model $\cB_m$, and sequences close to  $\dv^*\in\D$ (as per Step 1 in the template above). We make the following definitions. Let
$$\Delta_1= 2 \Delta(\dv^*)   + n^{\eps/6}$$
and define  $\D^+$ 
to be the set of all sequences $\dv\in \Z^n$ 
with $\Delta(\dv)\le    \Delta_1$ and $M_1(\dv)=2m$. 
For an integer $r\ge 0$, denote by $Q^0_r$ (or $Q^1_r$) the set of all even (or odd, respectively) sequences in $\Z^n$ that  have $L_1$
distance at most $r$ from  some sequence in $\D^+$. (Recall that we defined the parity of $\dv$ to be the parity of $M_1(\dv)$.)
We will estimate the ratio of the probabilities of adjacent degree sequences in the random graph model using the following. Define $R_{ab} $ as in Section~\ref{s:prelim} with $\cA = {V\choose 2}$. 

\begin{claim}\thlab{RSparse}
Uniformly for  all sequences $\dv\in  Q^1_1$ and for all $a,b\in [n]$
\begin{equation}
R_{ab}(\dv)=
\frac{ d_a}{ d_b} \left(1 +\frac{( d_a- d_b)(M_1+M_2)}{M_1^2}  \right) \left(1+O\left( \frac{\Delta_1^6}{d^3n^2}\right)\right), 
\ee
where 
$M_1=M_1(\dv)$ and 
$M_2 =M_2(\dv) = \sum_{v\in[n]} d_v( d_v-1) $. 
\end{claim}

 \begin{proof} 
First, let $\dv\in  \D^+$, let $n_1$ be the number of non-zero coordinates in $\dv$.  
By summing vertex degrees we find $dn=M_1(\dv)\le n_1\Delta(\dv)$. 
By assumption we therefore have 
\bel{deltas}
\Delta(\dv)\le \Delta_1 \ll d^{1/3}n^{1 /6} = \frac{(dn)^{1/3}}{n^{1/6}}\leq \frac{\left(\Delta(\dv) n_1\right)^{1/3}}{n^{1/6}},
 \ee
which readily implies that $\Delta(\dv) =o(n_1^{1/2}/n^{1/4})$. Lemma~\ref{l:graphical}\eqref{l:graphicalii} applied to the sequence formed by the  non-zero coordinates of $\dv$ now implies that 
$\Num({\dv})>0$  for $n$ sufficiently large. We can deduce the same conclusion for all $\dv\in Q^0_{8}$, since $\Delta(\dv)$ and $n_1(\dv)$ can only change by bounded factors when moving from such $\dv$ to the closest member of $\D^+$. 
Similarly, $M_1(\dv)=\sum d_i=dn+O(1)$ for all $\dv\in Q^0_{8}$. It is now clear by~\thref{l:simpleSwitching} and~\eqn{deltas} that
\bel{Pbound}
P_{av}(\dv)=O( \Delta_1^2/ dn) = o(1) \quad  \mbox{for all $\dv\in Q^0_{8}$ and all $a\neq v$}.
\ee 
Next consider any distinct $a,v,b\in[n]$ and   $\dv\in Q^0_6$, with $d_a>0$ and $d_v>0$. Then $\dv-\ea-\ev\in Q^0_{8}$ and hence $ \Num(\dv-\ea-\ev)>0$ from above, and also $P_{av}(\dv-\ea-\ev)=O( \Delta_1^2/dn)<1$ 
using~\eqn{Pbound}. Thus, for $n$ sufficiently large, $ \Num_{av}(\dv)>0$ by Lemma~\ref{trick17}, and we have  $ \Num_{av}(\dv)<   \Num (\dv) $ since $P_{av}(\dv)<1$ for similar reasons. 
 This establishes the  hypotheses for $\Num_{av}(\dv)$ and $\Num(\dv)$ in \thref{l:recurse} 
whenever they are needed below. 

 It now follows that $\YPavb_{avb}(\dv)=O( \Delta_1^4/ d^2n^2)$  for $\dv\in Q^0_6$; if $d_a$ or $d_v$  is 0 then this is immediate, and otherwise it follows from \thref{l:recurse}(c) in view of~\eqn{Pbound}, and noting that the numerator is non-negative by definition. 
Next, definition~\eqn{eq:bad} yields $\Bad(a,b,\dv)=O( \Delta_1^4/ d^2n)$ for all distinct $a,b\in [n]$ and all $\dv\in Q^0_6$ with $d_a>0$.  (In the current setting  $ \cA(i)\setminus \cA(j)|=\{j\}$  when $i\ne j$, and $d\le \Delta_1$.)   
Thus~\eqn{Ratformula} gives $R_{ab}(\dv)=d_a/ d_b(1+O( \Delta_1^4/ d^2n))$ for all $\dv\in Q^1_5$ and all distinct $a, b$ such that 
$d_a,\,  d_b> 0$. If now $\dv\in Q_4^0$ and $d_v >0$, we have $\Num_{bv}(\dv)>0$ as noted above. 
Therefore 
$\sum_{b\in \cA^*(v)} d_b = M_1(\dv) -d_v$ for such $\dv$. Thus, Proposition~\ref{l:recurse}(a) gives \bel{Pavsparse}
P_{av} (\dv) = d_a  d_v/M_1(1+O( \Delta_1^4/d^2n))
\ee
 for all $\dv\in Q^0_4$ and $a\neq v$. 
 (If $d_a$ and $d_v$ are both nonzero, the proposition applies as mentioned above, and if either is 0, the claim holds trivially.) 
Using a similar argument, \thref{l:recurse}(c)   gives 
$$
\YPavb_{avb}(\dv)=\frac{d_a[d_v]_2 d_b}{M_1^2}\left(1+O\left(\frac{\Delta_1^4}{d^2n}\right)\right)
$$
for all $\dv\in Q^0_2$ and all distinct $a,v,b\in [n]$.  
Applying these results to the definition~\eqn{eq:bad}  of $B$ for  $\dv\in Q^0_2$ and distinct $a,b\in[n]$, 
and recalling that  $d\le \Delta(\dv)+2$, now gives the sharper estimate
$$
\Bad(a,b,\dv)=\bigg(\frac{d_b}{M_1}  + \frac{ d_bM_2}{M_1^2}\bigg)\bigg(1+O\bigg(\frac{ \Delta_1^4}{d^2n}\bigg)\bigg),
$$
which we note is $O(\Delta^2/d n)$ as $M_2\le M_1\Delta $, where $\Delta$, $M_1$ and $M_2$ are with respect to $\dv$. 
Thus, for all $\dv\in Q_1^1 $ and  all $a,b$, and noting that $M_2$ changes by  a negligible additive term $O(\Delta_1)$ under bounded perturbations of the elements of the sequence $\dv$,
\bel{Req1}
  R_{ab}(\dv) =\frac{d_a}{ d_b}\frac{\big(1 -  (d_b-1)/M_1 -  ( d_b-1)M_2/M_1^2 \big)}
{\big(1 -   (d_a-1)/M_1 -  ( d_a-1)M_2/M_1^2\big)}\bigg(1+O\bigg( \frac{\Delta_1^6}{d^3n^2}\bigg)\bigg),
\ee
which implies the claim. 
\end{proof}
 
 We  next  make the definitions of probability spaces necessary to apply Lemma~\ref{l:lemmaX} (see Steps 2 and 3 in the template). 
 Let $\Omega$
be  the underlying set of $\cB_m(n)$, $\W=\D^+$ and $\cS'=\cD(\G(n,m))$. Let $H(\dv)=\pr_{\cB_m}(\dv) \widetilde H(\dv)$, where 
$$ \widetilde H(\dv) = \exp\left( \frac{1}{4}-\frac{\gtwo^2}{4\mu^2(1-\mu)^2}\right),$$
and define  the probability function in 
$\cS$ by 
\bel{PS}
 \pr_\cS(\dv)=    H(\dv)/\sum_{ \dv'\in\Omega}H(\dv') = \frac{ H(\dv)}{\ex_{\cB_m} \widetilde H}.
\ee
Let $G$ be the graph with vertex set $\W$ and with two vertices (sequences) adjacent if they are of the form $\dv-\ea$, $\dv-\eb$ for some  $a,b\in[n]$  and odd $\dv$. 

We need to estimate the probability of $\W$ in the two probability spaces (see Step 4 in the template). In  $\G(n,m)$ each  vertex degree  is distributed hypergeometrically with expected value $d=2m/n$. 
Also note that, letting  $\Delta^*=\Delta(\dv^*)$, we have by definition $\Delta_1\ge \Delta^*+n^{\eps/12}\sqrt {\Delta^*}\ge d+n^{\eps/12}\sqrt {\Delta^*}$. 
Thus, for   $\dv\in  \Omega$, 
$$
 \pr_{\cS'}(d_i > \Delta_1)\le  \pr_{\cS'}\big(d_i>d+ n^{\eps/12 }\sqrt {\Delta^*}\big)    =o(n^{- \omega})
$$ 
by~\cite[Theorems~2.10 and 2.1]{JLR} for example (and noting $\Delta^*\to\infty$).
 The union bound, applied to each $i$, now gives $\pr_{\cS'}(\W) =  1- o(n^{- \omega})$. 
 
For similar reasons $\pr_{\cB_m} (\W) =  1- o(n^{-\omega})$. To deal with the exponential factor $\widetilde H(\dv)$, 
we claim that if $\dv$ is chosen according to $\cB_m(n)$ then  
$\gtwo (\dv) = \mu(1-\mu)(1+O(\xi))$ with probability $1-o(n^{-\omega})$, where $\xi = O(n^{\eps-1/2})$. 
Indeed, this follows from the forthcoming \thref{l:sigmaConc}(ii) in which we may take $\alpha = \log^2n/\sqrt{n}$ and note that $\log^3n=o(dn)$ is implied by $m/\sqrt{n}\to\infty$. %
 Thus, for such $\dv   \in \cB_m(n)$, the exponential factor $\widetilde H(\dv)$ is 
 $1+O(\xi)$ with probability $1- o(n^{-\omega})$, 
 and it is always at most $e^{1/4}$. We deduce that  
$\ex_{\cB_m} \widetilde H=  1+O(\xi)$  and additionally, $\pr_{\cS} (\W) = \ 1- o(n^{-\omega})$.
 Thus, we may  set $\eps_0= O(1/n)$ in \thref{l:lemmaX} (with apologies to the function $n^{-\omega}$, ending its life in this proof dominated by $1/n$).
   
To apply \thref{l:lemmaX} (see Step 5 in the template above), the final condition we need to show is that the ratios of probabilities satisfy 
\bel{eq:target1}
\frac{\Pr_{\cS'}(\dv-\ea)}{\Pr_{\cS'}(\dv-\eb)} =e^{O(\delta)}\frac{\Pr_{\cS}(\dv-\ea)}{\Pr_{\cS}(\dv-\eb)}
\ee
whenever $\dv-\ea$ and $\dv-\eb$ are elements of $\D^+$, for a particular  $\delta=\delta(\Delta_1)$   independent of $\dv$ and specified below, where  the constant implicit in $O()$ is independent of $\dv$ and $\dv^*$.

To evaluate the right hand side of~\eqref{eq:target1} we observe that 
\begin{align}\label{eq:ratioSparse}
\frac{\Pr_{\cS}(\dv-\ea)}{\Pr_{\cS}(\dv-\eb)} 
= \frac{H(\dv-\ea)}{H(\dv-\eb)}&=   \frac{d_a(n- d_b)}{d_b(n -d_a)}\exp\left(\frac{ (d_a-d_b)   \gamma_2  }{   d^2 (1-\mu')^2 } +O\left({{\Delta}^2\over (dn)^2}\right) \right), 
\end{align}
 for all  $\dv\in Q_1^1$ 
by~\eqref{conjRatio},
where $d$, $\gtwo$  and $\Delta$ are defined with respect to $\dv$, and $\mu'=\mu(\dv-\ea)$. 
Note that   $\mu' = d(1+O(1/nd))/(n-1)$  and   
$d\leq \Delta +1$,  since $\dv\in Q_1^1$.  
 Thus, we also have $\gtwo=(M_2+M_1- d M_1)/(n-1)^2
=O(d \Delta  /n)$. Hence,
the argument of the exponential factor in~\eqn{eq:ratioSparse} is 
$$
\frac{(d_a-d_b)\gamma_2 }{d^2} +O( \Delta^2/n^2)  
=\frac{ (d_a-d_b)(M_2+M_1- dM_1)}{M_1^2}+O( \Delta^2/n^2).
 $$
 Combining this with~\eqn{eq:ratioSparse} and \thref{RSparse} it follows that for all $\dv\in Q_1^1$  
 \bel{ratios}
 R_{ab}(\dv) = \frac{H(\dv-\ea)}{H(\dv-\eb)}\bigg(1+O\bigg( \frac{\Delta_1^6} {d^3n^2}\bigg)\bigg),
\ee
 where we use that $ (n-d_b)/(n-d_a)=\exp\big((d_a-d_b)/n+O(\Delta^2/n^2)\big)$,  
$M_2=O(\Delta M_1)$ 
and $\Delta^4/M_1^2\le \Delta^6/d^3n^2$, $\Delta\le \Delta_1$, and the most significant error term derives from \thref{RSparse}. 

Equation~\eqn{ratios} now implies that  
$$
\frac{  \pr_{\cS'}(\dv-\ea)}{ \pr_{\cS'}(\dv-\eb) } = R_{ab}(\dv) = e^{O(\delta)}\frac{  \pr_{\cS}(\dv-\ea)}{ \pr_{\cS}(\dv-\eb)}
$$
whenever $\dv-\ea$ and $\dv-\eb$ are elements of $\W=\D^+$, 
where we may take 
$
\delta= \Delta_1^6/d^3 n^2.
$

 It is clear that the diameter of $G$ is at most $r:= m =nd/2$. Lemma~\ref{l:lemmaX}  then   implies  that $P_{\cS'}(\dv)=e^{O( r\delta+\eps_0)}P_{\cS}(\dv)$    for $\dv\in \D^+$. 
To proceed from here,  since we found that   $\ex_{\cB_m} \widetilde H =1+O(\xi)$, equation~\eqn{PS} implies $P_{\cS}(\dv) = H(\dv)(1+O(\xi))$ for $\dv\in \D^+$.  Hence, 
\bel{generalError}
P_{\cS'}(\dv^*)=e^{O(r\delta+\eps_0+\xi)}  H(\dv^*). 
\ee
Note that $\xi=O(n^{\eps-1/2})$, and  $r\delta+\eps_0=  O( \Delta_1^6/d^2n)=  O( (\Delta(\dv^*)^6 +{n^\eps})/d^2n)$. The theorem follows since $\cS'=\cD(\G(n,m))$.     
\end{proof}

 The result in Conjecture~\ref{conj1} or~\eqn{formula} follows from this in the sparse case, with different error terms, as long as $d$ is appreciably above $1/\sqrt n$. For smaller $d$, the analysis could be adjusted to obtain results, however the random graph is quite uninteresting here, typically having most vertices of degree 0, and the rest of degree 1 except for perhaps   a few vertices of degree 2.

  We note that the above result applies in the  case of $d$-regular graphs only for $d=o(n^{1/4})$, far short of  $o(\sqrt n)$ as reached in~\cite{MWlow}. It is also quite straightforward to reach past $\sqrt n$ using our method, by carrying the  calculations  a little further, iterating  several more  the recursive equations that are only used  twice in the proof above. In fact, this is how we first obtained the formulae  for $P$ and $R$ in later sections. Having  derived those ``limiting" formulae, our proofs can completely avoid considering the  iterated versions of the formulae, as shown in the next section.

\section{Function operators and fixed points }\lab{sec:FunOps}
In the previous section we used two iterations of the recursive equations from \thref{l:recurse}, each time applying them to all degree sequences at a certain distance from a
root sequence $\dv$. This allowed us to determine the ratio $\Num(\dv-\ea)/\Num(\dv-\eb)$ 
up to negligible error terms. 
For denser graphs we would need  an unbounded number of 
iterations to obtain the desired precision  of about $O(1/n \sqrt{d} )$  since the improvement is $O(\Delta/n)$ each time. 
Instead of doing this, we define operators based on the recursive identities 
from \thref{l:recurse} and study their behaviour on input functions that are close to the desired functions.

Let  $\Z^n$  denote the set of non-negative integer sequences  of length $n$.   
 For a given integer $n$ and a set $\cA\se \binom{[n]}{2}$ we define $\oA$ to be the set of ordered pairs $(u,v)$ with $\{u,v\}\in \cA$.  Ordered pairs are needed here because, although the functions of interest are symmetric in the sense that the probability of an edge $uv$ is the same as $vu$, our approximations to the probability do not obey this symmetry. Similarly, let $\oA_2$ denote the set of ordered triples $(u,v,w)$ with $u$, $v$ and $w$ all distinct and $\{u,v\},\{v,w\}\in \cA$. 
                                  
Suppose we are given  $\pv:  \oA  \times \Z^n \to  \reals_{\ge 0}$,  $\yv :  \oA _2 \times \Z^n \to\reals_{\ge 0}$    and  $\rv : [n]^2 \times \Z^n \to\reals_{\ge 0}$. We  write $\pv_{av}(\dv)$ for $\pv(a,v,\dv)$ (where $\dv\in \Z^n $), and  remind 
the reader that in this paper, a subscript  $av$   always denotes an ordered pair rather than an edge. 
Similarly, we write  $\yv_{avb}(\dv)$ for $\yv(a,v,b,\dv)$ and $\rv_{ab}(\dv)$ for $\rv(a,b,\dv)$.  
  We  also define  an  associated function  $\bad({\pv,\yv})$ as follows. 
For $\dv\in \Z^n $ and $a,b\in [n]$  with $a\neq b$, set 
$\bad(\pv,\yv)(a, a, \dv) = 0$ and 
\begin{equation}\lab{def:bad}
\bad({\pv,\yv})(a, b, \dv)  = 
\frac{1}{d_{a}}
\left( 
\sum_{v\in \cA(a)\sm \cA(b)} \pv_{av}(\dv) 
+ \sum_{v\in \cA(a)\cap \cA(b)} \yv_{avb}(\dv) \right).
\end{equation}
We define   operators 
$\Pc(\pv,\rv)$, $\Yc(\pv,\yv)$   and $\Rc(\pv,\yv)$, acting on $\pv$, $\yv$ and $\rv$ as above, 
  as follows.   
For $\dv \in \Z^n$ 
and $a,v,b \in [n]$,  we set 
\begin{align}
\Pc(\pv,\rv)(a,v,\dv)  & = 
d_{v} \left(\sum_{b\in \cA(v)}
 \rv_{ba} ({\bf d} - \ev) 
\frac{1-\pv_{bv}({\bf d} - \eb - \ev)}
{1-\pv_{av}({\bf d} -\ea - \ev)}\right)^{-1}  \text{ for } (a,v)\in \oA, \lab{F1def}\\
 \Yc(\pv,\yv)(a,v,b,\dv)  & =   
  \frac{\pv_{av}(\dv)\cdot\big(\pv_{bv}(\dv-\ea-\ev) - \yv _{avb}(\dv-\ea-\ev)\big)}{1-\pv_{av}(\dv-\ea-\ev)}   \text{ for } (a,v,b)\in \oA_2
,  \lab{FQdef}\\
\Rc(\pv,\yv)(a,b,\dv) &= 
\frac{d_{a}}{d_{b}}\cdot \frac{1- \bad(\pv,\yv)(a,b,\dv-\eb)}{1- \bad({ \pv,\yv})(b,a, {\bf d} - \ea)}. \lab{F2def}
\end{align}

\thref{l:recurse} says that in a certain sense, the probability and ratio functions $P_{av}$   $\YPavb_{avb}$ and $R_{ab}$ are fixed points of the operators $\cP$, $\Yc$ and $\Rc$. It is very useful for us that these operators are ``contractive", in a certain sense, in a neighbourhood of this fixed point. Unfortunately, the concept of contraction which we have here uses a slightly different metric before and after applying the operators, stemming from the fact that the value of $\Rc(\pv,\yv)$ at a point $(a,b,\dv)$ depends on values   of $\pv$ and $\yv$ at points $(c,w,\dv')$ and $(c,w,h,\dv')$ for several   $\dv'$ in a neighbourhood of $\dv$. This makes it difficult to define a true and useful  contraction mapping. 
Nevertheless, we can exploit the useful features of the situation using the following lemma.

\begin{definition}\thlab{def:Pi-set} 
Let $\D_0\se\Z^n $  and let $\mu \in \reals$. We use  $\Pi_\mu(\D_0)$  to denote the set of  pairs of functions $(\pv,\yv)$ with $\pv : \oA \times \Z^n \to  \reals_{\ge 0}$ and $\yv :  \oA _2 \times \Z^n \to\reals$ 
  such that for all even $\dv\in \D_0$, we have
 \begin{itemize}
 \item[$(\Pi a)$]
 $0\le\pv_{av}(\dv)   \le \mu$ for all $ (a,v) \in \oA$, 
 \item[$(\Pi b)$]
 $\sum_{v\in \cA(a)\cap\cA(b)} \yv_{avb}(\dv)\le  \mu d_a $ for all $a\ne b\in [n]$,  and 
 \item[$(\Pi c)$]
$0\le \yv_{avb}(\dv)\le     \mu   \pv_{bv}(\dv)$ for all  $ (a,v,b)\in  \oA_2$.
\end{itemize}
\end{definition} 
 
We denote by $Q_s^0(\dv),\, Q_s^1(\dv)\, \se\mathbbm{Z}^n $  the set of even and odd, respectively,  
  vectors of arbitrary integers  
that have  $L_1$ distance at most $s$ from $\dv$.  
Recall that we use $1\pm \xi$ to denote a quantity in the interval $[1-\xi,1+\xi]$.

\begin{lemma}\thlab{l:errorImplication} 
There is a constant $C>0$ such that the following holds. Let~$n$  be an integer and $\cA\se \binom{[n]}{2}$. 
Let $\dv=\dv(n)\in\Z^n$ satisfy $|\cA(a) \setminus  \cA(b)|<  d_a$   whenever $\cA(a)\cap\cA(b)\ne \emptyset$. 
Let $0<\xi \le 1$ and $0<\mu_0=\mu_0(n)<C$. Let $(\pv,\yv)$ and $(\pv', \yv')$ be members of $\Pi_{\mu_0}( Q_{2}^0(\dv))$, and let
 $\rv, \rv':[n]^2\times \Z^n \to \reals$.
 Let $a,v,b\in [n]$. 
  \begin{itemize}
\item[(a)]  
If $\dv$ is odd,  $\cA(a)\cap\cA(b)\ne \emptyset$, $\pv_{cw}(\dv')=\pv'_{cw}(\dv')(1 \pm\xi )$ for all  $(c,w)\in \oA$  and all $\dv' \in  Q^0_{  1}(\dv)$, and
 $\yv_{cwh}(\dv')=\yv'_{cwh}(\dv')(1 \pm\xi )$ for all  $ (c,w,h)\in  \oA_2 $  and all $\dv' \in  Q^0_{1}(\dv)$,
     then 
$$
\Rc({\pv,\yv})_{ab}(\dv)= \Rc({\pv',\yv'})_{ab}(\dv)(1+O (\mu_0\xi  )). 
$$
\item[(b)]  
If $\dv$ is even, $(a,v)\in \oA$, $\pv_{cv}(\dv')=\pv'_{cv}(\dv')(1 \pm\xi )$ for all $c \in \cA(v) $ and all $\dv' \in  Q^0_{2}(\dv)$, and $\rv_{ca}(\dv')=\rv_{ca}'(\dv')(1 \pm\mu_0\xi)$ for all $c \in \cA(v)$ and all $\dv'\in  Q^1_1 (\dv)$,  then  
$$
\Pc({\pv,\rv})_{av}(\dv)=\Pc({\pv',\rv'})_{av}( \dv)\left(1+O (\mu_0\xi  )\right). 
$$
\item[(c)]  
If $\dv$ is even, $(a,v,b)\in \oA_2$, $\pv_{cv}(\dv')=\pv'_{cv}(\dv')(1 \pm\mu_0\xi  )$ for all $c \in \cA(v) $ and all $\dv' \in  Q^0_{2}(\dv)$, and  $\yv_{cwh}(\dv')=\yv'_{cwh}(\dv')(1 \pm\xi )$ for all  $(c,w,h)\in  \oA_2 $  and all $\dv' \in  Q^0_{2}(\dv)$,
then 
$$
\Yc({\pv,\yv})_{avb}(\dv)=\Yc({\pv',\yv'})_{avb}(\dv) (1 +O\left(\mu_0 \xi \right) ). 
$$
\end{itemize}
The constants implicit in $O(\cdot)$ are absolute. 
  \end{lemma}
\thref{l:errorImplication} indicates that applying the operators once (one after the other) to functions that are close in terms of relative error yields functions that are closer by a factor of $\mu_0$, where in our applications of the lemma, $\mu_0$ will be a (very) loose upper bound on the density of the graphs of interest. We define suitable distance functions below to make this idea precise. 
We also remark that in (a) we restrict to $\cA(a)\cap \cA(b)\neq\emptyset$ since otherwise no vertex $v$ satisfies $avb\in \oA_2$. This will be useful for an application to bipartite graphs in a later paper. Note that when $\cA = \binom{[n]}{2}$ we have $\cA(a)\cap \cA(b)\neq\emptyset$ for all $a,b \in [n]$.

\begin{proof}  For (a), since $\pv_{cw}(\dv')=\pv'_{cw}(\dv')(1 \pm\xi )$ and  
 $\yv_{cwh}(\dv')=\yv'_{cwh}(\dv')(1 \pm\xi )$, from~\eqn{def:bad} and the non-negativity of the functions in $\Pi_{\mu_0}$, we obtain that $\bad({ \pv,\yv})(a,b,\dv') =\bad({\pv',\yv'})(a,b,\dv')\big(1 +O(\xi )\big)$    for all $\dv'\in  Q_1^0 (\dv)$.  
Additionally for such $\dv'$, the assumption that $(\pv',\yv')\in\Pi_{\mu_0}( Q_{2}^0(\dv))$, together with the assumption that  $|\cA(a) \setminus  \cA(b)|<  d_a$ for $\cA(a)\cap\cA(b)\neq \emptyset$, imply that
$\bad(\pv',\yv')(a, b, \dv') =   O( \mu_0)$.    Thus
$\bad( \pv,\yv)(a, b, \dv') =\bad(\pv',\yv')(a, b, \dv')  +O(\xi\mu_0)$  for all $\dv'\in Q_1^0(\dv)$. 
Hence, with $\mu_0$ sufficiently small to ensure $\bad(\pv,\yv)(b,a,\dv-\ea)<1/2$ say, 
part~(a) follows from~\eqn{F2def}.

The equation for $\Pc(\pv,\rv)$ in (b) follows similarly from~\eqn{F1def} since both $\pv_{bv}$  and $\pv_{av}$ are  bounded by $\mu_0$, so in particular the denominator in~\eqn{F1def} is bounded away from 0. Finally, (c) follows easily via~\eqn{FQdef} and noting the bound of $\mu_0  \pv_{bv}(\dv')$ on $ \yv _{avb}(\dv')$.
\end{proof}

Fix $\Omega^{(0)}\se \Z^n $  for the following definitions. Let $ \Omega^{(s)}$ denote the set of all $\dv\in \Omega^{(0)} $   for which  $Q_s^0(\dv),\, Q_s^1(\dv) \se \Omega^{(0)}$. We will make use of these sets as restricted domains for functions of $\dv$ that refer to slightly altered sequences $\dv'$. 
Define a set of distance functions, indexed by $s$, on the set of all $ (\pv,\yv) $ for which  $\pv:  \oA  \times  \Z^n \to \reals_{\ge 0}$ and $\yv:  \oA_2  \times  \Z^n \to \reals_{\ge 0}$ by 
\begin{align}
\chi^{(s)} \big((\pv,\yv ), (\pv',\yv')\big) &= \max 
\left\{\chi_1 ^{(s)} (\pv,\pv'),\chi_2 ^{(s)} ( \yv,\yv') \right\}, \qquad \mbox{where}\\
\chi_1 ^{(s)} (\pv,\pv') &= \sup 
\left\{\big|\log\big( \pv_{cw}(\dv)/\pv'_{cw}(\dv)\big)\big|:  (c,w)\in \oA,\  \dv\in \Omega^{(s)} \right\},\\
\chi_2 ^{(s)} ( \yv,\yv') &= \sup 
\left\{\big|\log\big( \yv_{cwh}(\dv)/\yv'_{cwh}(\dv)\big)\big|:  (c,w,h)\in \oA_2,\  \dv\in \Omega^{(s)} \right\}.\end{align}
 If any denominator is 0 we define $\chi_i^{(s)}$, $i=1,2$, to be $\infty$ at that point, so these are extended metrics. Clearly $\chi^{(s)}$ is non-increasing in $s$ unless $\Omega^{(s)}$ is  empty, in which case we set $\chi^{(s)}=0$. 
 Also, define the compositional operator
$$\Cc(\pv,\yv)=\big(\hat \pv,\Yc(\hat \pv,\yv)\big) $$
where $\hat \pv=\Pc\big(\pv,\Rc(\pv,\yv)\big)$.

\begin{lemma}\thlab{c:contraction}
There is a constant $C>0$ such that the following holds. 
Let  $n$  and $\cA$ be 
as in \thref{l:errorImplication}, 
let $\Omega^{(0)}\se \Z^n $. 
Let $ \xi \le 1/2$ and $0<\mu_0=\mu_0(n)<C$.   
Assume that $s\ge 0$ such that  $(\pv,\yv)$,   $(\pv',\yv')  \in  \Pi_{ \mu_0}\big( \Omega^{(s)} \big)$, and 
  $\chi^{(s)} \big((\pv,\yv ), (\pv',\yv')\big)\le \xi$. Then $\chi^{(s+ 4)}(\cC(\pv,\yv), \cC(\pv',\yv')) =O(\mu_0\xi)$.
\end{lemma}

\proof
Let $\dv\in \Omega^{( s+1)}$  be odd. Then $ Q_{2}^0(\dv)=  Q_{1}^0(\dv)\subseteq \Omega^{(s)}$, and so   $(\pv,\yv)$, $(\pv', \yv')\in \Pi_{\mu_0}( Q_{2}^0(\dv))$, as required for the lemma. Since $\chi^{(s)} \big((\pv,\yv ), (\pv',\yv')\big)\le \xi \le 1/2 $,   we have $\pv_{cw}(\dv')=\pv'_{cw}(\dv')(1\pm 2\xi)$ for all   $(c,w)\in \oA$  and  $\dv' \in  Q^0_{1}(\dv)$,  and a similar statement holds for $\yv$, $\yv'$. 
 We can thus apply \thref{l:errorImplication}(a) and, defining $\rv=\Rc(\pv,\yv)$ and $\rv'=\Rc(\pv',\yv')$, deduce that $ 
\rv_{ab}(\dv)= \rv'_{ab}(\dv)(1+O (\mu_0\xi  ))  
$ for all such $\dv$ and for $a,b$ as in that lemma. Preparing for the next step, note that the error term $O (\mu_0\xi  )$ is $\pm \xi$ for $C$ sufficiently small. 

Now let $\dv\in \Omega^{( s+  2)}$ be even and define $\hat \pv=\Pc(\pv,\rv)$ and $\hat \pv'=\Pc(\pv',\rv')$. 
Applying~\thref{l:errorImplication}(b)  in a similar way 
gives 
$$
\hat\pv_{cw}(\dv)=\hat\pv'_{cw}(\dv)\big(1+O(\mu_0 \xi)\big) 
$$
for  $(c,w)\in  \oA$, and again we may assume the error term is $\pm \xi$.

Finally, we may repeat the process with $\dv\in \Omega^{( s+ 4)}$,   and use \thref{l:errorImplication}(c)  to deduce that
$$
\Yc(\hat \pv,\yv)_{avb}(\dv)=\Yc(\hat \pv',\rv')_{avb}(\dv) (1 +O\left(\mu_0 \xi \right) ) 
$$
for all $(a,v,b)\in\oA_2$ for such $\dv$. 
 Since the error terms $\mu_0\xi < C/2$ in the last two conclusions can be made arbitrarily small by taking $C$ small,  we have $ \log(1+O(\mu_0 \xi))=O(\mu_0 \xi)$ as required to deduce the lemma. \qed

\section{Proof of the binomial model in the graph case}\lab{s:completinggraphs}

 In this section we prove \thref{conj2} for $p$  in the ``gap range" which we can describe as $o(n^{-1/2})<p< c/\log n$.  
Before doing so we need concentration results for some functions $f(\dv)$ when $\dv$ 
has either independent binomial entries, or is the degree sequence of $\G(n,m)$.  
In~\cite[Theorem 3.4]{degseq1} it was essentially shown that, 
when $\dv$ has independent binomial entries,     
   $\sigma^2=\sigma^2(\dv)= \sum_{i=1}^n(d_i-d)^2/n$  is concentrated. 
   We give a more efficient proof of the crucial  part of this, 
   using the following  result, which 
we will also apply to the degree sequence of $\G(n,m)$. This is a direct corollary of Theorem (7.4) and Example (7.3) of McDiarmid~\cite{McD}. However, since the constants there are not explicit and the framework makes the proof not so easily accessible, we give a proof here.
\begin{lemma}[McDiarmid]
\thlab{l:subsetConc}  Let $c>0$ and let $f$ be a function  defined on the set of   subsets of some set $U$ such that $\size{f(S)-f(T)}\le c$ whenever $|S|=|T|=m$ and $|S\cap T|=m-1$. Let $S$ be a randomly chosen $m$-subset of $U$.   Then   for all $\alpha>0$ we have
$$
\pr\left(\size{f(S)-\ex f(S)} \ge \alpha c\sqrt m  \right) \le 2\exp (- 2\alpha^2).
$$
 \end{lemma}
\proof  
Consider a  process in which the random subset $S$ is generated by inserting $m$ distinct elements one after another, each randomly chosen from the remaining available ones. Let $S_k$ denote the $k$th subset formed in this process, $0\le k\le m$. Consider the Doob martingale process determined by $Y_k=\ex\big(f (S_m) \mid S_0,\ldots, S_k  \big)=\ex\left(f(S_m) \mid  S_k\right)$. Given $S_{k-1}$, let ${\bf X} = X_k,\ldots, X_m$ denote the remaining elements added in the process.   Let ${\bf X}_0$ be the random sequence ${\bf X}$  conditioned on $X_k=x_k\in U$, and ${\bf X}_1$  the random sequence   ${\bf X}$ conditioned on $X_k=x_k'\in U$. Then ${\bf X}_0$ can be coupled with ${\bf X}_1$      by interchanging $x_k$ and $x_k'$ wherever they occur in ${\bf X}_0$. The   values of $f(S_m)$ in the two elements of a couple pair differ by at most $c$ by assumption. Since each possibile realisation of   $X_k,\ldots, X_m$   has the same probability, it follows that 
$$
\size{\ex\left(f(S_m)\mid  S_{k-1}\wedge X_k=x_k \right)- \ex\left(f(S_m)\mid  S_{k-1}\wedge X_k=x_k'\right)}\le c,
$$
 and hence 
$|Y_{k-1}-Y_k|\le c$. Azuma's Inequality (see, e.g., \cite{JLR}), or alternatively~\cite[Corollary (6.10)]{McD}, now completes the proof.
\qed

  Recall that by $\omega$ we denote a function that tends to $\infty$ arbitrarily slowly with $n$, and that $\cB_m (n)$ is  a sequence of $n$ i.i.d.~random variables each distributed as ${\rm Bin}(n-1,p)$ conditioned on $\sum d_i=2m$. 
\begin{lemma}
\thlab{l:sigmaConc}  Define $\dv=(d_1,\ldots, d_n)$ as either (a) the degree sequence of a random graph in $\G(n,m)$, or   (b)  a sequence in $\cB_m (n)$. Let $d=2m/n$. Then  
\begin{enumerate}[(i)]
\item for $1\le i\le n$ and  all $\alpha>0$   we have 
$$
\Pr(|d_i-d|\ge \alpha  )\le 2\exp\bigg(-\frac{\alpha^2}{2(d+\alpha/3)}\bigg);
$$ 
\item  if $\log^3 n=o(dn) $ and $\alpha$ satisfies $(\log n)/\sqrt n +(\log^{3/2} n)/\sqrt{dn}=o(\alpha)$ then 
 we have 
$$\pr(|\sigma^2 - \Var\, d_1| \ge \alpha d +1/n) = o(n^{-\omega }).$$ 
Moreover, $ \Var\, d_1 = d(n-d)/n+O(d/n)$.  
\end{enumerate}

\end{lemma}

\begin{proof} 
We deal with the graph case (a)  first. 
Each  vertex degree $d_i$ is distributed hypergeometrically with parameters ${n\choose 2},m,n-1$, expected value $d$, and hence  
(i) holds by~\cite[Theorems~2.10 and~2.1]{JLR}. 
For a graph $G$ with degrees $d_1,\ldots, d_n$   define
$$
f = f(G)=\sum_{i=1}^n \min\{(d_i-d)^2,  x\}, 
$$
where $x>1$ is specified below.
Then increasing or decreasing the value of a single $d_j$  by 1 whilst holding $d$ fixed can only change $f$ by at most $ (\sqrt{ x})^2 - (\sqrt{x}-1)^2 <2\sqrt x$. Since $G$ is determined by a random $m$-subset of all possible edges,    Lemma~\ref{l:subsetConc} applies with $c=8\sqrt x$ (as each edge in the symmetric difference of $S$ and $T$ affects two vertex degrees). Replacing $\alpha$ appropriately gives
\begin{align*}
\pr(|f(G)-\ex f(G)|  \ge \alpha dn) &\le 2\exp (- \alpha^2 dn  / (32x))\\
  & =o(n^{-\omega })
\end{align*}
provided that  $\alpha^2 dn/(x\log n)\to\infty$. 
On the other hand, 
let    $A$ denote the event $\max_i|d_i-d|\ge \sqrt x$. By (i) and the union bound applied over all $n$ values of $i$, we have   $\pr(A)= o(n^{-\omega })$ as long as we choose $x=\omega(d\log n +\log^2 n)$.  By the bound on $\alpha$, there exists $x$ satisfying both conditions, and at this point we set $x$ as such. 
Provided that $A$ does not hold,  we have $f(G)= \sum(d_i-d)^2=n \sigma^2$. We thus conclude
$$
\pr\left(\size{\sigma^2 -\ex f(G)/n}  \ge \alpha d\right) =o(n^{-\omega}) +O(1)\pr(f(G)\ne  n\sigma^2  )=o(n^{-\omega}).
$$
Now evidently
$$
\size{\ex f(G) - n \ex  \sigma^2}  =O(n^3) \pr(f(G)\ne  n \sigma^2 )=o(n^{-\omega}) 
$$
and thus 
$$
\pr\left( \size{\sigma^2 - \ex \sigma^2}  \ge \alpha d+ 1/2n\right) =o(n^{-\omega}).
$$
Noting that $\ex (d_i-d)^2=\Var\,  d_1 $, we obtain part (ii) for (a). The estimate for $\Var\,  d_1$ follows from the standard formula for variance of this hypergeometric random variable.

For the binomial random variable case (b), essentially the same argument applies for both (i) and (ii), 
 by regarding $ d_1,\ldots, d_n $ each as a sum of $n-1$ independent indicator variables. Conditioning on  the sum  being $2m$ is equivalent to a uniformly   random selection of a $2m$-subset of the $n(n-1)$ indicator variables.
 \end{proof}

We shall see that the following  establishes \thref{conj2}  
in the gap range with explicit error terms. 
 Recall that  given a sequence $\dv$ we write $M_1=M_1(\dv)=\sum_i d_i$, $d=d(\dv)=M_1/n$, $\mu=\mu(\dv)=d/(n-1)$, and $\sigma^2(\dv) = \frac{1}{n} \sum_{i=1}^n(d-d_i)^2$ where $n$ is the length of the sequence $\dv$.

 \begin{thm}\thlab{t:handCalc}
Let $n$ and $m$ be integers, 
and assume that $\mu_1 =   2m/n(n-1)$ satisfies $(\log n)^K/n < \mu_1 =o\big(1/ (\log n)^{3/4} \big)$
for all fixed $K>0$. 
Let $\D$ be the set of sequences $\dv$ of length $n$ satisfying the following for some constant  $C\ge 2$:  
\begin{enumerate}[(i)]
\item $M_1(\dv) = 2m$ (and thus $\mu = \mu(\dv)=\mu_1$  and  $d=d(\dv)=2m/n$), 
\item $\size{d_i -   d} \le C \sqrt{ d\log n} $ for all $i\in[n]$, 
\item $\sigma^2(\dv) \leq 2  d$.  
\end{enumerate} 
Then 
\begin{enumerate}[(a)]
\item in each of the models $\cB_m(n)$ and $\cD(\G(n,m))$  we have 
$\Pr(\D)=1-n^{-h(C)}$, where $h(x)\to\infty$ as $x\to\infty$, and 
\item for $\dv=\dv(n)\in\D$ we have 
	$$ \Pr_{\cD(\G(n,m))}(\dv) 
	= \Pr_{\cB_m}(\dv) \exp\Big( {1\over 4} - {\gtwo^2\over 4\mu^2(1-\mu)^2} \Big) 
	\left(1+O\left( \frac{1}{\sqrt{ d}}+ \frac{ d\sqrt{\log n}}{n}
	+ \frac{   d^2(\log n)^{3/2}}{n^2}
	 \right)\right),$$
	where $\gtwo=\gtwo(\dv) = \frac{1}{(n-1)^2}\sum_i(d_i- d)^2$.
\end{enumerate}
\end{thm}
\noindent We note that the constant implicit in $O()$ in (b) of course can, and in fact does, depend on $C$. 
\begin{proof} 
Let $\Omega$ be the underlying set of $\cB_m(n)$ and let 
$\dv\in \Omega$. We will consider $\dv$ chosen either according to $\cD(\G(n,m))$ or $\cB_m(n)$. 
By definition, $M_1(\dv) = 2m$ for all $\dv\in \Omega$. 
Apply \thref{l:sigmaConc}(i) with $\alpha=C\sqrt{d \log n}$ and the union bound to see that  
 in both $\cD(\G(n,m))$ and $\cB_m(n)$,
 with probability at least $1-n^{-f(C)}$, for all $i\in [n]$ we have 
$\size{d_i-d} \leq C\sqrt{d\log n}$ where we may take $f(C) = C^2/3-1$. 
Now, apply \thref{l:sigmaConc}(ii) with $\alpha=1/2$ and note that $  d =\mu_1 (n-1)>\log^2 n$  by the theorem's hypothesis to obtain that $\sigma^2(\dv) \leq 2d$ with probability at least $1-n^{-\omega}$. 
Therefore, $\dv$ satisfies~(ii) and~(iii) with probability $1-n^{-h(C)}$, 
for some function $h(C)\to\infty$ as $C\to\infty$, and $\dv$ satisfies 
(i) always. Hence (a) follows.

For (b) we first consider the ratio for adjacent degree sequences (see Step 1 in the template given at the start of Section~\ref{s:sparseGraphs}).
Let  $Q_1^1$ be the set of sequences $\dv\in\Z^n$ 
such that $\dv-\ea\in \D$ for some $a\in[n]$. 
Recall that $P_{av}(\dv)$ denotes the probability that the edge $av$ is present in a graph 
 $G\in\cG(\dv)$ 
 and that 
$$R_{ab}(\dv)= \frac{\Pr_{\cS'}(\dv-\ea)}{\Pr_{\cS'}(\dv-\eb)},$$
where $\cS' = \cD(\G(n,m))$. 
We now present functions $\Pgr$, $\Rgr$ and $\Ygr$ that approximate the probability and ratio functions $P$, $R$ and $Y$ sufficiently well. 
For $a,v,b\in [n]$ set 
\begin{align}
\Pgr_{av}(\dv) &= \frac{d_ad_v}{d(n-1)}\left(1-\frac{(d_a-d)(d_v-d)}{d(n-1-d)}\right), \lab{Pgrdef}\\
\Rgr_{ab}(\dv) &= \frac{d_a(n-d_b)}{d_b(n-d_a)}
	\left(1+ \frac{(d_a-d_b)}{d^2n}\sigma^2(\dv)\right), \lab{Rgrdef} \\
\Ygr_{avb}(\dv) 
	&=P^{\mathrm{gr}}_{av}(\dv ) P^{\mathrm{gr}}_{bv}(\dv -\ea-\ev)(1+1/n). \lab{Ygrdef}
\end{align}
 
\begin{claim}\thlab{RDense}
Uniformly for $\dv=\dv(n) \in \D$ and $a\ne v\in  [n]$ 
\bel{Peq}
P_{av}(\dv)=\Pgr_{av}(\dv)\big(1+O(\eta_1+\eta_2 )\big), 
\ee
and uniformly for all $\dv\in Q_1^1$ 
and for all  $a, b\in  [n]$ 
\bel{Req}
 R_{ab}(\dv)=\Rgr_{ab} (\dv)\big(1+O(\eta_1+ \mu \eta_2  )\big),
\ee
where 
 $\eta_1 = 1/m + \eps \mu_1 /n$
 and 
 $\eta_2 = \eps/n +  \eps^3\mu_1^2$, 
with $\eps = C\sqrt{(\log n)/d}$, where $d=2m/n$. 
\end{claim}
The proof will show that an analogous statement also holds for $Y_{avb}$. Note that $\eps$ is simply the upper bound on  the relative degree spread of a sequence $\dv$ in $\D$ implied by (ii).

\begin{proof}[Proof of \thref{RDense}] 

To show that $P$ and $\Pgr$ (and $R$ and $\Rgr$) are $(\eta_1+\eta_2)$-close in the sense of \eqref{Peq} and \eqref{Req}, we consider the operator $\cC$ as defined for \thref{c:contraction}. We first observe that $\cC$ fixes $(P,Y)$, where in this context  we regard $P$ to be the function $\pv$ with $\pv_{ av}=P_{av} $ for all appropriate $a$ and $v$, and similarly $Y$  to be   $\yv$ with $\yv_{avb}= Y_{avb}$, by \thref{l:recurse}. We will next use the contraction property of $\cC$ as expressed in \thref{c:contraction} to show that for any integer $k>0$, $\cC^{k}(\Pgr,\Ygr)$ and $\cC^{k}(P,Y)$ are $O(\mu)^{k }$-close. We will also show that $\Pgr$ and $\cC^{k}(\Pgr)$ are $(\eta_1+\eta_2)$-close. These observations will then be shown to imply \thref{RDense}.

Fix $k_0 = 4\log n$ and $r=4k_0+4=O(\log n)$.   
Let $\Omega^{(0)}$ be the set of sequences $\dv  \in \Z^n$ that are at $L_1$ distance at most $r$ from a sequence in $\D$.  Let $\mu_0=5\mu_1$, and  define $\Omega^{(s)}$ as in \thref{c:contraction}  to be the set of sequences $\dv\in\Omega^{(0)}$ of $L_1$ distance at least $s+1$ from all sequences outside   $\Omega^{(0)}$. 

Towards \thref{c:contraction} we first establish that $(P,Y)$ and $(\Pgr,\Ygr)$ are elements of $\Pi_{\mu_0}(\Omega^{(2)})$ (see \thref{def:Pi-set}).
 Note that for $\dv\in \Omega^{(0)}$, the values of $d(\dv)$ and $\mu(\dv)$ are asymptotically equal to $2m/n$ and $\mu_1$, respectively, since $M_1(\dv)=M_1(\dv_0)+O( \log n)$ for some sequence $\dv_0\in \D$. Thus, $\mu$ ($=\mu(\dv)$) and $\mu_1$ are interchangeable in the error terms below, as are  $1/dn$ (where $d=d(\dv)$) and $1/m$. 
 Furthermore, we note that condition~(ii) of the theorem, together with the lower bound on $\mu_1$ in the theorem statement with say  $K=2$, imply that for all even $\dv \in \Omega^{(0)}$, $d_i\sim \mu n$ uniformly for all~$i$. 
To bound $P_{av}(\dv)$, we first observe, using $\mu\geq  (\log n)^3/n$ and assumption (ii) of the theorem, that Lemma~\ref{l:graphical}\eqref{l:graphicali} implies
$\Num({\bf d})>0$ for all even  $\dv \in \Omega^{(0)}$.
After this, for $n$ sufficiently large,~\thref{l:simpleSwitching}, together with 
 the fact that $d_i\sim \mu n$ uniformly for all $i$,
implies that for all distinct $a,v\in [n]$ 
\bel{Pbound2}
P_{av}(\dv)\leq\frac{\mu}{1-\mu}(1+o(1)) <  \frac{5\mu}{4} 
\quad  \mbox{for all even $\dv\in \Omega^{(0)}$},
\ee 
where for the last inequality we use that $\mu\sim\mu_1 < 1/6$, say.  
Since $\Omega^{(2)}\se \Omega^{(0)}$ and $5\mu/4 <\mu_0$ this establishes requirement $(\Pi a)$ for $P$ in the definition of $\Pi_{\mu_0}(\Omega^{(2)})$. 
Now restrict slightly to $\dv \in \Omega^{(2)}$. 
 By definition $Y_{avb}(\dv)$ is the probability that both edges $av$ and $bv$ are present.
Hence  \thref{l:recurse} (c)implies (with the above bounds on $P_{av}(\dv)$ applying for all $\dv\in \Omega^{(0)}$) that $0\le Y_{avb}(\dv) =Y_{bva}(\dv) \le 3\mu_1 P_{bv}(\dv)/2$ 
  (easily) assuming, as we may, that $\mu_1$ is sufficiently small. Thus $(P,Y)$ 
 satisfies condition $(\Pi c)$ for membership of $\Pi_{\mu_0}(\Omega^{(2)})$, and also
 $$
 \sum_{v\in [n]\setminus\{a,b\}} Y_{avb}(\dv) \le  \sum_{v\in [n]\setminus\{a,b\}}\frac{3P_{bv}(d)\mu_1}{2}\le 3 \mu_1d_a(1+o(1)) 
 $$
using $P_{bv}(\dv)\le 2\mu_1\sim 2d_a/n$
 which follows from \eqref{Pbound2} and $\mu\sim\mu_1$.  As  $4\mu_1 < \mu_0$, this shows $Y$ satisfies condition $(\Pi b)$ for membership of $\Pi_{\mu_0}(\Omega^{(2)})$ when $n$ is sufficiently large. 
To see that $(\Pgr,\Ygr)$ is also in $\Pi_{\mu_0}(\Omega^{(2)})$  we recall that $d_i\sim \mu n$ uniformly for all $i$ for all $\dv \in \Omega^{(0)}$.  Thus, by definition~\eqn{Pgrdef} we have $\Pgr_{av}(\dv)\sim \mu$ for all distinct $a,v\in [n]$ and all $\dv\in \Omega^{(0)}$. 
Properties $(\Pi a)$-$(\Pi c)$ follow directly from this fact and the definition~\eqref{Ygrdef} since $\mu\sim \mu_1 <\mu_0/5$. 

 Now for large $n$ and distinct $a,v\in [n]$ we have $P_{av}(\dv)=\Pgr_{av}(\dv)(1\pm1)$ for all even $\dv\in \Omega^{(0)}$ since $\Pgr_{av}(\dv)\sim \mu$ and by~\eqref{Pbound2}.  Also,  $0\le Y_{avb}(\dv)\le 3\mu_1P_{bv}(\dv)/2$ implies  $Y_{avb}(\dv)=\Ygr_{avb}(\dv)(1\pm1)$ even $\dv\in \Omega^{(2)}$.  
We may now apply \thref{l:errorImplication}(a) with $\xi=1$ for any odd $\dv \in \Omega^{(3)}$  to deduce that
$$
\Rc(P ,Y ) _{ab}(\dv)= \Rc(\Pgr ,\Ygr )_{ab}(\dv)(1+O (\mu_0  ))
$$
for all $a,b\in [n]$. 
Writing $\rv$ for $\Rc(P ,Y )$ and $\rv'$ for $\Rc(\Pgr ,\Ygr )$ we obtain from this 
and \thref{l:errorImplication}(b) that 
$$
\Pc(P ,\rv ) _{av}(\dv)= \Pc(\Pgr ,\rv' )_{av}(\dv)(1+O (\mu_0  ))
$$ 
for all even   $\dv \in \Omega^{(4)}$, all distinct $a,v\in [n]$.  
Next applying \thref{l:errorImplication}(c) in the same way to even $\dv \in \Omega^{(6)}$ gives 
$$
\Yc(\hat\pv ,Y) _{avb}(\dv)= \Yc(\hat\pv' , \Ygr )_{avb}(\dv)(1+O (\mu_0  ))
$$ 
for all even $\dv \in \Omega^{(6)}$, all distinct $a,v,b\in [n]$,
where $\hat\pv = \Pc(P ,\rv )$ and $\hat\pv' = \Pc(\Pgr ,\rv' )$. 
Recalling the definition of the distance functions $\chi^{(s)}$ and the definition of the operator $\cC$ we see that this is equivalent to the statement 
$ \chi^{(6)}\big( \cC(P ,Y ),\cC(\Pgr ,\Ygr )\big) =O(\mu_0)$. 
Making $k-1$ iterated applications of~\thref{c:contraction} with ever-decreasing $\xi$ produces
$$
\chi^{(4k+2)}\big(\cC^{k_0 }(P ,Y ),\cC^{k_0 }(\Pgr ,\Ygr )\big) =  O(\mu_0)^{k}.
$$
Finally, 
 $\Cc(\Pgr ,\Ygr ) (\dv) $ can be estimated by straightforward expansions  using the following. Uniformly for all $\dv \in \Omega^{(0)}$:
\begin{enumerate}[(a)] 
\item $\Rc(P^{\mathrm{gr}},Y^{\mathrm{gr}})_{ab}(\dv)= R^{\mathrm{gr}}_{ab} (\dv)\left(1 + O\left(\eta_1\right)\right)$
	for   all $a,b \in  [n]$;  
\item $\Pc(\Pgr,\Rgr)_{av}(\dv) = \Pgr_{av}(\dv)\left(1 + O\left(\eta_2\right)\right)$ 
	for   all  distinct  $a,v \in  [n]$; 
\item $\Yc(\Pgr,\Ygr)_{avb}(\dv) = \Ygr_{avb}(\dv)\left(1 + O\left(\eta_2\right)\right)$ 
	for  all  distinct $a,v,b \in  [n]$. 
\end{enumerate}
This is justified by \thref{l:hand} in Appendix \ref{s:technical}, after making two observations. One is that  the error term  $\eps^4 d^2/n^2$ can be dropped because $\eps^3 d =o(1)$.  The second is that the error terms   $\eta_1$ and $\eta_2$ are now defined with reference to sequences which are at distance $O(\log n)$ from the sequences in $\D$ and are thus  asymptotically the same as the stated values. (In fact, a better approximation is proved in \thref{l:mapleHigher} using computer assistance.) 
Applying  (a), (b) and (c) in turn,   recalling Lemma~\ref{l:errorImplication} to handle the small error terms, shows that  $\chi^{(6)}(\Cc(\Pgr,\Ygr),(\Pgr,\Ygr)) = O(\eta_1+ \eta_2)$. 
 Using \thref{c:contraction} repeatedly, and bounding the total distance moved during the iterations as for a contraction mapping  (as the sum of a geometric series), this gives
$$
\chi^{(r-2)}\big((\Pgr ,\Ygr ),\cC^{k_0 }(\Pgr ,\Ygr )\big) = O(\eta_1+ \eta_2). 
$$
Combining this with above bound on $\chi^{(4k+2)}$ when $k=k_0$, and  with   $\cC(P,Y)=(P,Y)$ and the triangle inequality, gives
$
\chi^{(r-2)}\big( (P ,Y ), (\Pgr ,\Ygr )\big) =  O( \eta_1+ \eta_2)+ O(\mu_0)^{k_0}. 
$
This implies~\eqn{Peq} for all even $\dv\in \Omega^{(r-2)}$ since
$k_0 = 4 \log n$, and  we may assume  $ O(\mu_0)<1/e$ say since $\mu_0=5\mu_1\to 0$. 
Note that $\D\se  \Omega^{(r)}\subseteq \Omega^{(r-2)}$ by definition.  
 For~\eqn{Req}, we now use that~\eqn{Peq} holds for all even $\dv\in \Omega^{(r-2)}$ to deduce from \thref{l:errorImplication}(a) that 
$\Rc(P,Y)_{ab}(\dv) = \Rc(\Pgr ,\Ygr )_{ab}(\dv) (1+O(\mu_0\eta_1+\mu_0\eta_2))$ for all odd $\dv\in \Omega^{(r-1)}$. 
This, together with (a) above and the facts that $\Rc(P,Y)=R$ and $\mu_0=O(\mu)$, implies~\eqn{Req} for all $\dv\in Q_1^1 \subseteq   \Omega^{(r-1)}$.
\end{proof}

 Since $\eps = C\sqrt{(\log n) / d}$  and $d<n/\sqrt {\log n}$, the claim  gives
\bel{targetLHS}
\frac{\Pr_{\cS'}(\dv-\ea)}{\Pr_{\cS'}(\dv-\eb)} = 
R_{ab}(\dv)= 
\Rgr_{ab}(\dv)
\left(1+O\left(  \frac{1}{dn} + \frac{ (\sqrt{d  \log  n})^3}{n^3} 
 	+\frac{ \sqrt{d  \log  n}}{n^2}\right)\right) 
\ee
uniformly for all $\dv\in Q_1^1$.

We now move to Steps 2 and 3 in the template in Section 4. Let 
$
H(\dv)=\pr_{\cB_m}(\dv) \widetilde{H}(\dv)
$
be the conjectured formula in the right hand side of (b) (without error terms), 
where 
$$ \widetilde{H}(\dv)= \exp\Big( {1\over 4} - {\gtwo^2\over 4\mu^2(1-\mu)^2} \Big).  
$$
Define the probability spaces $\cS$ and $\cS'$ exactly as in the proof of Theorem~\ref{t:sparseCase} with the same underlying set $\Omega$.  
That is, 
$$\pr_{\cS}(\dv)=H(\dv)/\sum_{\dv'\in \Omega}H(\dv')=\frac{H(\dv)}{\ex_{\cB_m}\widetilde{H}}$$ and $\cS'=\cD({\G(n,m)})$. 
 Also define the  graph $G$  as before, with vertex set $\W: = \D$, and with an edge joining each two   sequences in $\D$ of the form $\dv-\ea$ and $\dv-\eb$ for some $a\ne b$. 
The $L_1$ distance from a sequence $\dv\in G$ to the constant sequence $(d,\ldots,d)$ 
is $\sum_i \size{d_i-d}$, which is at most $n\sqrt{2d}$ by (iii) and Cauchy's Inequality. 
Some vertex of $G$ has $L_1$ distance at most $n$ from this constant sequence. It follows that the diameter of $G$ is $r=O(n\sqrt d)$.

We claim (see Step 4 of the template) that $\W$ has probability at least $1-\eps_0$ for some suitably chosen $\eps_0$ in both $\cS$ and $\cS'$. 
Note that $\Pr_{\cS'}(\W) = 1-n^{-h(C)}$ and $\Pr_{\cB_m}(\W) = 1-n^{-h(C)}$ by (a) proved above. 
Furthermore, if $\dv\in \cB_m(n)$ then $\gtwo(\dv) = \frac{n}{(n-1)^2}\sigma^2(\dv) = \mu(1-\mu)(1+O(\xi))$ with probability 
$1-o(n^{-\omega })$, where $\xi=\log^2 n/\sqrt n$ (this is the more precise implication of \thref{l:sigmaConc}(ii) 
applied with $\alpha = \log^2n/\sqrt{n}$).
Thus, for such $\dv$ in $\cB_m(n)$, the exponential factor $\widetilde{H}(\dv)$ is $1+O(\xi)$ with probability $1-o(n^{-\omega })$. Therefore,
\bel{exHtilde}
\ex_{\cB_m}\widetilde{H} =1+O(\xi)
\ee 
  and thus $\pr_{\cS}(\W) = 1-n^{-h(C)}$. 
It follows that, as in the proof of \thref{t:sparseCase}, we may use $\eps_0=n^{-1}$ in \thref{l:lemmaX}. 

We now move to Step 5 in the template.
For  $\dv\in Q_1^1$,
\begin{align}\lab{eq:ratioDense}
\frac{\Pr_{\cS}(\dv-\ea)}{\Pr_{\cS}(\dv-\eb)}
&= \frac{H(\dv-\ea)}{H(\dv-\eb)}
=   \frac{d_a(n- d_b)}{d_b(n -d_a)}\exp\left(\frac{ (d_a-d_b)   \gamma_2  }{d^2 (1-\mu')^2 } +O\left({{\Delta}^2\over (dn)^2}\right) \right)  
\end{align}
by~\eqref{conjRatio}, where $\mu'=\mu(\dv-\ea)$ and $\gtwo$, $d$ and $\Delta$ are defined with respect to $\dv$ as in~\eqref{conjRatio}.
Note that 
the first term in the exponential in~\eqn{eq:ratioDense} is 
$O\left( \sqrt{ \log n/dn^2}\right)$ 
 for $\dv\in Q_1^1$ by (ii) and (iii) and since $\mu'= \mu(1+1/2m)$. 
 Thus, 
 $$ 
 \frac{ (d_a-d_b)   \gamma_2  }{ d^2 (1-\mu')^2 } 
=\frac{ (d_a-d_b)}{ (dn)^2} \sum_i(d_i-d)^2 + O\left(\frac{\sqrt{d\log n}}{n^2}\right),
$$
and the second term, $\Delta^2/(dn)^2$, is $O(1/n^2)$ by assumption (ii). We can now infer from the definition of $\Rgr$ that~\eqref{eq:ratioDense} is equivalent to 
\bel{targetRHS}
\frac{\Pr_{\cS}(\dv-\ea)}{\Pr_{\cS}(\dv-\eb)}
= \Rgr_{ab}(\dv)  \left(1+ O\left(\frac{\sqrt{d\log n}}{n^2}\right)\right). 
\ee
This together with
\eqref{targetLHS} gives
$$
\frac{\Pr_{\cS'}(\dv-\ea)}{\Pr_{\cS'}(\dv-\eb)} =e^{O(\delta)}\frac{\Pr_{\cS}(\dv-\ea)}{\Pr_{\cS}(\dv-\eb)}
$$
with 
$$\delta =   \frac{1}{dn} + \frac{ (\sqrt{d  \log  n})^3}{n^3} 
 	+\frac{ \sqrt{d  \log  n}}{n^2}$$  
 for $\dv\in Q_1^1$.  
Therefore, by \thref{l:lemmaX} and~\eqn{exHtilde}, 
\begin{align*}
\Pr_{\cS'}(\dv) &= \Pr_{\cS}(\dv) e^{O\left( r\delta+\eps_0 \right)}\\
&= H(\dv) \left(1+O\left(\xi + r\delta+\eps_0\right)\right) 
\end{align*}
 for all $\dv\in \D$, 
which proves (b) since $\xi =  (\log^2n )/\sqrt{n}$, 
$r\delta = O\left(d^{-1/2}+ d\sqrt{\log n}/n+ d^2(\log n)^{3/2}/n^2\right)$ and $\eps_0=1/n$.
\end{proof}
It is a simple exercise in analysis to see that the theorem implies Conjecture~\ref{conj2} in the gap range: the truth of the theorem itself implies a slightly altered version of the theorem's statement (b), in which $C$  is a function  of $n$ that tends (``slowly'') to $\infty$. (The same can be done with the constant in the $O(\cdot)$ if desired.) The fact that $h(C)\to\infty$ then shows that the asymptotic approximation~\eqn{formula} holds for the sequences $\dv$ in a suitable set $R_p(n)$. All that remains is to note that the distribution of $m$ in $\cD(\Gnp)$ is identical to that in $\Epp$, and that the latter restricted $\sum d_i=2m$ is identical to $\cB_m(n)$.
\begin{cor} Conjecture~\ref{conj2}  holds.\thlab{cc}
\end{cor}

We remark that one can avoid the sharp concentration results that we used, instead employing only variance via Chebyshev's inequality, at the expense of relaxing the  $o(n^{-\omega})$ error in the conjecture  to $o(1)$. The  result would  still be interesting; we leave the details to the reader.

  
  

\section{A wider range of degrees: proof of \thref{t:graphwide}} \lab{s:wider}

 In this section we prove Theorem~\ref{t:graphwide}. Compared with \thref{t:handCalc}, some crucial  differences that affect the argument  include   $\mu$ being permitted to have constant size, the allowable degree spread being  $d^\alpha$  for $\alpha>1/2$, and the transfer of  $\sigma^2$ from explicit bounds to a term in the ratio formula.

The proof has the same structure as for Theorem~\ref{t:handCalc}.  The crucial change required is to redefine the approximations, $\Pgr$, $\Rgr$   and $\Ygr$,  
 of the probability and the ratio functions so that the error functions corresponding to $\eta_1$ and $\eta_2$ in  \thref{RDense} satisfy $\eta_1+\mu \eta_2 = o(1/(nd^{\alpha}))$. This error bound is necessary to obtain a final   formula with $(1+o(1))$ error, since with the range of degrees under consideration, the diameter of the graph $G$ of degree sequences  (see the proof of~\thref{t:handCalc}) is  up to  $r= O(nd^{\alpha})$.

To define the approximations, we  write $\Pgr$, $\Rgr$  and $\Ygr$  parametrised to facilitate identifying negligible terms.  
We define the expressions 
\begin{eqnarray*}
\pi & = &\mu (1+\eps_a)(1+\eps_v)
		\left(1 + \frac{-\mu\eps_a\eps_v+(\eps_a +\eps_v )\sigma^2  /d n}{1-\mu}
		+  \frac{\eps_a +\eps_v}{n-1}\right),\\
\rho  
	& = &\frac{1+\eps_a}{1+\eps_b}\cdot\frac{1-  \mu(1 + \eps_b)+1/n}
	{1-  \mu(1 + \eps_a)+1/n}  \left(1 
	+\frac{(\eps_a-\eps_b) \sigma^2}
	{(1-\mu)^2 dn } 
	\right).
\end{eqnarray*}
When referring to $\pi$, we list only an initial segment of parameters $\eps_a,\eps_v,  \mu,\sigma^2,d$   that are different from the ones in the definitions above.  So for instance $\pi(x,y)$ stands for $\pi$ with  $\eps_a,\eps_v$ replaced by $x,y$. Similarly for $\rho$ and the parameters $\eps_a,\eps_b,  \mu,\sigma^2,d$.
Recall that we consider  sequences $\dv$ of length $n$ and $\mu =  d/(n-1)$.

For this section, we define  $\Pgr$ and $\Rgr$ by 
$$
 \Pgr_{av}=\pi, \quad \Rgr_{ab} = \rho 
$$
 for all $a,v,b\in [n]$, 
where $\eps_i=( d_i-d)/d$ and $\sigma^2 = \sigma^2 ( \dv) = \sum (d_i-d)^2/n$, 
so that $P_{av}^{\mathrm{gr}}$ etc.\ are functions of degree sequences.  Furthermore, for all $a,v,b\in [n]$ we define 
$$\Ygr_{avb}(\dv) = \pi(\eps_a,\eps_v) \pi(\eps_b,\eps_v-\delta)
	\cdot\left(1+\frac{1+\eps_a - \mu(1+\eps_a+\eps_b)}{(n-1)(1-\mu)}\right).$$

Note also that $d$ and $\mu$ were  specified in the theorem statement (determined by $m$), but with a slight notational abuse, for the following lemma, given any sequence $\dv$ of length $n$  we
define $\mu=\mu(\dv) = \frac12 M_1(\dv) /   {n\choose 2}$  so that the average is $d = \mu (n-1)$. 
 In the following lemma, the parity of $\dv$ is immaterial, though it will only be applied for odd $\dv$ in (a) and even $\dv$ in~(b) and~(c). 

\begin{lemma}\thlab{l:mapleHigher} 
Let $n$ be an integer and let $1/2\le \alpha < 3/5$. Let $\cA=\binom{[n]}{2}$ and let $\dv=\dv(n)$ be a sequence of length~$n$ with average $\bar d$ 
	such that $\mu_1=\mu(\dv) = \bar d/(n-1) < 1/4$, 
	and assume that  for  all $1\le i \le n$ we have 
	$|d_i-\bar d| \leq \eps \bar d$, 
	where $\eps = \bar d^{\,\alpha -1} > 0$. 
 Then 
\begin{itemize}
\item[(a)] $\Rc(\Pgr,\Ygr) _{ab}( \dv)=  \Rgr_{ab} ( \dv)
 \left(1 + O\left(\mu_1\eps^4\right)\right)$  for all $a,b \in [n]$,  
\item[(b)] 
$\Pc(\Pgr,\Rgr)_{av}(\dv) =  \Pgr_{av}(\dv)
 \left(1 + O(\mu_1\eps^4)\right)$  for all distinct $a,v \in [n]$,  
\item[(c)] 
$\Yc(\Pgr,\Ygr)_{avb}(\dv) =  \Ygr_{avb}(\dv)
 \left(1 + O(\mu_1\eps^4)\right)$ for all distinct $a,v,b \in [n]$. 
\end{itemize}
\end{lemma}
%
%
%
%
%
\begin{proof} 
In the calculations below the following approximations of $\Pgr$ and $\pi$, respectively, will often be convenient. 
Let $\dv'$ be a sequence of length $n$ that is at $L_1$ distance $O(1)$ from $\dv$, and with $d_a'=d_a-j_a$ and $d_v'=d_v-j_v$. 
Here and in the following, the bare symbols $\mu$, $d$, $\eps_a$ and so on are defined with respect to the original sequence $\dv$, whilst $\mu'$, $d'$, $\eps_a'$, etc., are defined with respect to the average degree of $\dv'$. For such a sequence $\dv'$ we have that $\mu(\dv')$ is $\mu'= \mu_1+O(1/ n^2)$ since the average $d'$ of $\dv'$ is $\bar d + O(1/n)$. 
Therefore, the variable $\eps_a'$ defined as $(d_a'-d')/d'$ is equal to 
$\eps_a-j_a\delta+O(\mu\delta^2)$, where $\delta = 1/\bar d$, and the analogous equation holds for $\eps_v$. 
Similarly, $\sigma' = \sigma^2(\dv') = \sigma^2(\dv) + O(\eps\mu_1).$
Thus 
\begin{align}\lab{aux710}
P^{\mathrm{gr}}_{av}(\dv') 
&= \pi(\eps_a',\eps_v',\mu',\sigma',d')
=\pi(\eps_a-j_a\delta,\eps_v-j_v\delta)(1+O(\xi)),		
\end{align}
where here and below $\xi =  \mu_1\eps^4$ (note that $\eps\ge \bar d^{-1/2}$ by assumption so that $1/\bar dn =O(\mu_1\eps^4)$). 
In other words, the changes from $\mu_1$, $\bar d$, and $\sigma^2$ to 
$\mu'$, $d'$, and $\sigma'$ are negligible in the formula for $\Pgr$.

For (a), we note first that $\Rc( \Pgr,\Ygr)_{aa}(\dv) = 1 = \Rgr_{aa}(\dv)$ by definition of $\rho$ and $\Rc$ in~\eqref{F2def}. Assume now that $a\neq b$. Using~\eqn{F2def} to evaluate $\Rc( \Pgr,\Ygr)_{ab}(\dv)$, we estimate the expression 
$\bad(a,b,\dv-\eb)=\bad(\Pgr,\Ygr)(a,b,\dv-\eb)$ for which, in turn, we need to estimate 
$\sum \Ygr_{avb}(\dv-\eb)$, where the sum is over all $v\in  [n]$  such that both $av$ and $bv$ are allowable (see \eqref{def:bad}). 
By definition and~\eqref{aux710}, 
$$\Ygr_{avb}(\dv-\eb) = \pi(\eps_a,\eps_v) \pi(\eps_b-\delta,\eps_v-\delta)
	\cdot\left(1+\frac{1+\eps_a - \mu_1(1+\eps_a+\eps_b)}{(n-1)(1-\mu_1)} +O(\xi) \right),$$
where we use $\eps_a$, $\eps_b$ and $\mu_1$ in the third factor (rather than the altered versions $\eps_a'$ etc.) using the same reasoning as in the lead-up to~\eqref{aux710}.
Consider expanding this expression for $\Ygr_{avb}(\dv-\eb)$ ignoring terms of  order  $\eps^4$, and hence also ignoring  $\delta^2$ and $\eps^2\delta$, since $\eps^2\ge 1/\bar d$. A convenient way to do this is to make substitutions $\epsA_v=y_1\epsA_v$,   $ \delta  = y_1^2 \delta $, $\mu_1 = y_2 \mu_1$, $1/n =y_1^2y_2/n$, and so on (for instance, $\sigma^2/dn$ is $O(\eps^2\mu_1)$) where $y_1$ represents a parameter of size $O(\eps)$ and $y_2$ of size $O(\mu_1)$, and then expand about  $y_1=0$. 
We note by inspection that $\Ygr_{avb}(\dv-\eb)$ is of order  $\mu_1^2$, and the terms in its expansion hence have the corresponding upper bound  $O(\mu^2 \eps^i  )$ on their absolute sizes. 
We also note that in expanding a rational function about a nonsingular point, the error in the Taylor expansion is bounded by a multiple of the least significant terms omitted. This avoids any need to bound higher derivatives explicitly. In this way, expanding after these $\{y_1,y_2\}$ substitutions, and noting that $1/n=O(\mu_1\eps^2)$), we obtain
\begin{equation*}\Ygr_{avb}(\dv-\eb)
= J +O(\mu_1^2\eps^4), 
\end{equation*}
where $J$ is a polynomial of degree 3 in $y_1$. (Unfortunately $J$ is too large to write here.) Next, removing the `sizing' variables $y_i$ from $J$  by setting them equal to 1, and then expanding the result  about $\eps_v =0$ and retaining all terms of total degree at most 3 in $ \eps_v$, we get
\begin{equation*}
\Ygr_{avb}(\dv-\eb)
= c_{0} 
		+ c_{1}\eps_{v}
		+ c_{2}\eps_{v}^2
 		+O(\mu_1^2\eps^4),
\end{equation*}
where the functions  $c_{0}$, $c_1$, and $c_2$ 
are independent of $\epsV_v$. 
(By calculation, the third order term turns out to be absorbed by the error term.)
Then considering the definition of  $\bad(a,b,\dv-\eb)$  in \eqref{def:bad} we find that the second summation in that definition can be written as
\begin{align*}
\Sigma _{\bad} 
&:= \sum_{v\in \cA(a)\cap \cA(b)}\Ygr_{avb}(\dv-\eb)\\
&= \sum_{v\in \cA(a)\cap \cA(b)}
		(c_{0} 
		+ c_{1}\epsV_{v}
		+ c_{2}\epsV_{v}^2
		+ O(\mu_1^2\eps^4)) \\
&= n c_0 +  n\sigmaBSquared \delta^2c_{2}
	 	-2c_0 - c_{1}(\epsA_a+\epsA_b)
	-c_{2}(\epsA_a^2+\epsA_b^2)
  +O(n\mu_1^2\eps^4),  
\end{align*}
where in the last inequality we use   $\sum_{v\in [n]} \eps_v = 0$ and $\sum_{v\in [n]} \eps_v^2 = n \sigma^2 \delta^2$  (and we recall that the $\eps_v$'s are defined with respect to $\dv$). 
Noting that   $\cA(a)\setminus \cA(b)=\{b\}$, 
we can write $\bad(a,b,\dv-\eb)$ in~\eqn{F2def}, by using~\eqn{def:bad} and~\eqref{aux710}, as 
 \begin{align*}
\bad(a,b,\dv-\eb)
	&= \frac{1}{d_a} 
	\left(\Sigma_{\bad}
	+  \pi(\epsA_a,\epsV_b- \delta)
	  +O(\mu_1 \xi) 
\right), 
\end{align*}
where $d_a = (1+\epsA_a)\bar d$.   
Note that the error term from $\Sigma_{\bad}$ produces an absolute error term of size $O(\mu_1\eps^4)=O(\xi)$  in $\bad(a,b,\dv-\eb) $ since $n/d_a\sim 1/\mu_1$. Substituting the above expression, stripped of its error  terms, into
\begin{align*}
 \frac{\Rc(\Pgr)_{ab}(\dv)}{\rho} -1
 &= \frac{1}{\rho}\cdot \frac{(1+\epsA_a)(1-\bad(a,b,\dv-\eb))}{(1+\epsA_b)(1-\bad(b,a,\dv-\ea))}  -1
\end{align*}
and simplifying gives a rational function $\widehat F$. That is, $\Rc(\Pgr)_{ab}(\dv)/\rho  -1=  \widehat F + O(\xi)$.  After inserting the  size variables $y_1$ and $y_2$  into $\widehat F$ as specified above, and simplifying, we find it has $y_2$ as a factor (of multiplicity 1), and its denominator is nonzero at $y_1=0$.  Then expanding the expression in powers of $y_1$ shows that  $\widehat F=O( y_1^4)$. Along with the extra factor $y_2$, this implies $\widehat F=O(\xi)$. Thus, part (a) follows.

To prove part (b) note that, analogous to~\eqn{aux710}, 
if $\dv'=\dv -\ev$ 
we also have for $b\neq v$  
$$ \Rgr_{ab}(\dv')  = \rho(\epsA_a',\epsA_b', \mu', (\sigmaASquared)', d') 
=\rho\cdot (1 +O( \mu_1\eps^4 ))$$
 where we also use that $a\neq v$.  
Therefore, by definition \eqref{F1def}  and~\eqn{aux710},
\begin{align}\lab{aux1552}
\Pc( \Pgr, \Rgr )_{av}(\dv) 
	&= d_{v} \left(\sum_{b\in  \cA(v)}
		\Rgr_{ba} (\dv-\ev)
		\frac{1- \Pgr_{bv}(\dv  - \eb - \ev)}
		{1-\Pgr_{av}( \dv -\ea - \ev)}\right)^{-1}\nonumber\\
	&= d_{v} \left(\sum_{b\in \cA(v)} 
		\rho (\epsA_b,\epsA_a )  \cdot 
		\frac{1-\pi\left(\epsA_b- \delta ,\epsV_v- \delta\right)}
		{1-\pi\left(\epsA_a- \delta,\epsV_v- \delta\right)}
		\left(1 +O\left(\mu_1\eps^4\right)\right)\right)^{-1}. 
\end{align}
By expanding in $\epsA_b$ 
we obtain 
\begin{align*}
\rho(\epsA_b,\epsA_a) 
\cdot \frac{1-\pi(\epsA_b- \delta ,\epsV_v- \delta)}
			{1-\pi(\epsA_a- \delta ,\epsV_v- \delta )} 
			& = K + O(\eps^4) 
\end{align*} 			
where $K$ is a polynomial in $\epsA_b$  of   degree at most 3.  
Calculations using the  size variables $y_1$ and $y_2$  as above show  that $K = k_{0} + k_{1} \epsA_b  +  k_{2} \epsA_b^2 + O(\eps^4)$  for some $k_{i}$ independent of $\epsA_b$. 
 Also, recall that we have $\cA(v) = [n]\sm\{v\}$.    So
the main summation over~$b$  in~\eqn{aux1552}  can be evaluated  (noting again that $\sum_b \epsA_b^2 = n \sigmaA^2\delta^2$) as 
\begin{align*}
nk_{0} + n  \sigmaA^2\delta^2 k_{2} 
	 - k_{0} - k_{1}\epsA_v - k_{2} \epsA_v^2    
\end{align*}
with relative error $O(\eps^4)$, noting  
that $K$ has constant order, 
where we use that $\sum_{b\in [n]} \epsA_b = 0$.  
Using the size variables $y_1$ and $y_2$ as described above, we then find  that  $\Pc(\Pgr,\Rgr)_{av}(\dv) = \pi (1+O(\mu_1\eps^4))$ for $a\neq v$, with the extra factor $\mu_1$ arising in the error term in the same way as for $\Rc$ in part (a). Part (b) follows. 

For part (c) consider  
$Z = \Yc(\Pgr,\Ygr)_{avb}(\dv)/\Ygr_{avb}(\dv) -1.$ By definition of  $\Yc$, 
$\Pgr$, and $\Ygr$ we can write $Z$ as 
\begin{align}
Z&=\frac{\Pgr_{av}(\dv)\big(\Pgr_{bv}(\dv-\ea-\ev)-\Ygr_{avb}(\dv-\ea-\eb) \big)}
{\Ygr_{avb}(\dv) \big(1- \Pgr_{av}(\dv-\ea-\ev)\big) }
-1 
\end{align}
Then, using~\eqn{aux710}, replace $\Pgr_{av}(\dv)$ by $\pi(\eps_a,\eps_v)$, 
$\Pgr_{bv}(\dv-\ea-\ev)$ by $\pi(\eps_b,\eps_v-\delta)(1+O(\xi)),$ 
$\Pgr_{av}(\dv-\ea-\ev)$ by $\pi(\eps_a-\delta,\eps_v-\delta)(1+O(\xi)),$ 
and, using the same argument as leading to \eqn{aux710}, replace 
$\Ygr_{avb}(\dv-j\ea-j\eb)$ by 
$$\pi(\eps_a-j\delta,\eps_v-j\delta) \pi(\eps_b,\eps_v-(j+1)\delta)
	\cdot\left(1+\frac{1+\eps_a -j\delta - \mu_1(1+\eps_a-j\delta+\eps_b)}{(n-1)(1-\mu_1)}+O(\xi)\right)$$
for $j\in \{0,1\}.$ Now, as for part (a) and (b), use the sizing variables $y_1$ and $y_2$ and expand about $y_1=0$ (note that both denominator and numerator of the fraction in $Z$ are asymptotically equal to $\mu_1^2$). We then find that $Z$ is of size $O(\mu_1\eps^4)$, which proves $(c)$. 
\end{proof}

To prove \thref{t:graphwide},   we set $\W=\D$ and then follow the proof of \thref{t:handCalc}, referring to the set $\D$ within it as $\D'$. Since $\D'\subseteq \D$, \thref{t:handCalc}(a) implies $\pr(\W)=1-o(1)$ in both models
 $\cB_m(n)$ and $\cD(\G(n,m))$. Actually, in the current setting we permit higher values of $\mu$, which is now only bounded above by a small constant, but the earlier proof applies equally well for this extended range.

Next define $Q_1^1$ as in the proof \thref{t:handCalc} to be the set of sequences $\dv\in \Z^n$ such that $\dv - \ea \in \D$ for some $a \in [n]$. 
Then the proof of Claim~\ref{RDense} (which assumed only the upper bound $1/4$ on $\mu$) applies, with adjustment to   the error terms $\eta_i$  using Lemma~\ref{l:mapleHigher} in place of Lemma~\ref{l:hand}, to show that
\bel{Peqnew}
P_{av}(\dv)=\Pgr_{av}(\dv)\big(1+O( \mu_1\eps^4)\big),
\ee
\bel{Reqnew}
 R_{ab}(\dv)=\Rgr_{ab} (\dv)\big(1+O( \mu_1\eps^4)\big),
\ee
 uniformly for all $\dv \in \D$ (or  $\dv \in Q_1^1$, respectively),  and all appropriate $a$, $v$ and $b$. 
Using the latter together with the definition of $\rho$, we find in place of~\eqn{targetLHS}  (with the same definitions of $\cS$ and $\cS'$)  that 
\bel{targetLHSnew}
\frac{\Pr_{\cS'}(\dv-\ea)}{\Pr_{\cS'}(\dv-\eb)}
= \Rgr_{ab}(\dv)  \left(1+ O\left(\mu_1 \eps^4+1/n^{\, 2}\right)\right). 
\ee
 Then  for the independent binomial probability space $\cS$,~\eqn{eq:ratioDense} is only altered by the $\gamma_2$ term cancelling with the $\sigma^2$ term in the definition of $\pi$. Thus we obtain~\eqn{targetRHS} with error term reduced to $O(1/n^2)$. 
 
In this application of Lemma~\ref{l:lemmaX}, the diameter of the auxiliary graph $G$ is $r=O(nd^{\alpha})$ and $\delta= \mu_1 \eps^4+1/n^{\, 2}$ from~\eqn{targetLHSnew}. We may again use $\eps_0=1/n$. The result is
$$\Pr_{\cS'}(\dv) = \Pr_{\cS}(\dv) e^{O\left( r\delta+\eps_0 \right)} 
= H(\dv) \left(1+O\left(\xi + r\delta+\eps_0\right)\right),
$$
with $\xi = (\log^2n )/\sqrt{n}$ entering as before. The theorem follows, since $d^{5\alpha-3}\ge d^\alpha/n$.
\qed

\section{Concluding remarks} \lab{s:final}

In the main result, \thref{t:graphwide},  the upper bound $\mu_0$ on the density $\mu$ can probably be set equal to any constant less than $1/2$, at the expense of only slight changes to the proof. Since other results cover this range of density, we do not follow this line any further. On the other hand, various aspects of the proofs can  be improved with some straightforward work, to obtain a wider range of degrees and smaller error terms, and we plan to pursue this elsewhere.

The  approach of Sections~\ref{s:recursive} and~\ref{sec:FunOps} can be applied to other problems. We plan to apply the method in order to prove related binomial-based models for the degree sequences of random bipartite graphs, loopless directed graphs, and hypergraphs. The approach can also be used to make a major advance in asymptotic enumeration of Latin rectangles.


\bibliographystyle{pagendsort}
\bibliography{references}




\appendix

\section{Approximating the operator fixed points}
\lab{s:technical}

Here we prove the estimates (a), (b) and (c) inside the proof of \thref{RDense}. 
Recall the definitions of 
$\Pgr$, $\Ygr$ and $\Rgr$ in~\eqref{Pgrdef}-\eqref{Ygrdef}, and of $\sigma^2(\dv)$.    
Recall also that $\D$ is a set of sequences of length $n$ satisfying 
$|d_i-d|\le C\sqrt{d\log n}$ for all $i\in[n]$ (for some constant $C$) and 
that $\sigma^2(\dv)\le 2d$  by the  assumptions of \thref{t:handCalc}. 
Furthermore, recall that we set $k_0=4\log n$ and that $\Omega^{(0)}$ 
is the set of sequences $\dv\in\Z^n$ that are at $L_1$ distance 
at most $r=4k_0+4$ from a sequence in $\D$. 
Therefore, (a), (b) and (c) follow from the following lemma with $\eps = (C+o(1))\sqrt{\log n/d}$.

\begin{lemma}\thlab{l:hand} 
Let 
$n$ be an integer 
and let $\cA=\binom{[n]}{2}$. 
Let $\dv$ be a sequence of length~$n$ with average $d$ 
such that $d/(n-1) < 1/4$ and $\sigma^2(\dv)=O(d)$, and assume that 
for  all $1\le i \le n$ we have 
$\size{d_i-d} \leq \eps d $, where $\eps = \eps(n)>0$ is bounded above by a sufficiently small constant. 
Then 
\begin{enumerate}[(a)] 

\item\label{R-appr} $ \Rc(P^{\mathrm{gr}},Y^{\mathrm{gr}})_{ab}(\dv)= R^{\mathrm{gr}}_{ab} (\dv)\left(1 + O\left(\eta_1\right)\right)$ for all $a,b\in [n]$, 
\item\label{P-appr} $\Pc(P^{\mathrm{gr}},R^{\mathrm{gr}})_{av}(\dv) = P^{\mathrm{gr}}_{av}(\dv)\left(1 + O\left(\eta_2\right)\right)$ for all distinct $a,v\in [n]$, and 
\item\label{Q-appr} $ \Yc(\Pgr,\Ygr)_{avb}(\dv)= \Ygr_{avb} (\dv)\left(1 + O\left(\eta_2\right)\right)$ for all distinct $a,v,b\in [n]$,
\end{enumerate}
where 
$\eta_1 =  1/dn + \eps d/n^2 + \eps^4 d^2/n^2$ 
 and 
 $\eta_2 =  1/dn+ \eps/n +  \eps^3d^2/n^2$.
\end{lemma}

\begin{proof} For completeness, this proof duplicates some steps in the proof of Lemma~\ref{l:mapleHigher}. 
We first  reparametrise  $\Pgr$, $\Rgr$ and $\Ygr$ to facilitate identifying negligible terms. 
Write $\mu=\mu(\dv) $ for the  ``density"  $ d/(n-1)$ of a graph with degree  sequence  $\dv$, and note that, by assumption, $\mu\le 1/4$. Define 
the  sequence $(\eps_1,\ldots,\eps_n)$ of  relative deviations from the average degree, 
that is~$\eps_a =(d_a-d)/d$ for $1\le a \le n$. 
Note that 
$|\eps_a|\le \eps$ for each $a$.  
Set $\sigma^2=\sigma^2( \dv )$ and 
\begin{align*}
\pi(\eps_{a},\eps_{v},\mu)
	&=\mu (1+\eps_{a})(1+\eps_{v}) \cdot \left(1 - \frac{\eps_{a}\eps_{v}\mu}{1-\mu}\right), \\
\rho(\eps_{a},\eps_{b},\mu,\sigma^2,d,n)
	&=\frac{1+\eps_{a}}{1+\eps_{b}} \cdot \frac{1 - \mu (1+\eps_{b})}{1 - \mu (1+\eps_{a})}
	\cdot\left(1+(\eps_{a}-\eps_{b})\frac{\sigma^2}{dn}\right), 
\end{align*}
Note that with this definition of $\pi$ we have $P^{\mathrm{gr}}_{av}(\dv) = \pi $ for all distinct $a,v\in [n]$. 
As in Section~\ref{s:wider}, when referring to $\pi$ and $\rho$, we list only an initial segment of parameters containing all those that are different from the ones in the definitions above. 

In the calculations below the following approximations of $\pi$ will often be convenient. 
Let $\dv'$ be sequence that is at $L_1$ distance $O(1)$ from   $\dv$, and with $d_a'=d_a - j_a$ and   $d_v'=d_v - j_v$. 
For such a sequence, $\mu$ becomes $\mu'= \mu+O(1/ n^2)$ since $d$ changes by $O(1/n)$. Therefore, the variable $\eps_a$ changes to $\eps_a-j_a\delta+O(\mu \delta^2)$, where $\delta =1/d$, and the analogous equation holds   for $\eps_v$.
 (Here and in the following, the bare symbols $\mu$, $d$, $\eps_a$ and so on are defined with respect to the original sequence $\dv$, whilst $\mu'$, $d'$, $\eps_a'$ etc.\  are defined with respect to the average degree of 
$\dv'$.) 
Thus  
\begin{equation}\lab{aux112}
P^{\mathrm{gr}}_{av}(\dv')  = \pi(\eps_a',\eps_v',\mu') =\pi(\eps_a-j_a\delta,\eps_v-j_v\delta)(1+O(\mu\delta^2)), 
\end{equation}
from which we see that the small changes in $\eps_ b$ for $b\ne a,v$ have negligible effect.

For \eqref{R-appr}, if $a=b$ then $\Rc(\Pgr,\Ygr)_{aa}(\dv) =1 = \Rgr_{ab}(\dv)$ by definition;  see \eqn{F2def} and~\eqref{Rgrdef}. 
So assume $a\neq b$. Using~\eqn{F2def} to evaluate $ \Rc(\Pgr,\Ygr)_{ab}(\dv)$, we estimate the expression  
$\bad(\Pgr,\Ygr)(a,b,\dv-\eb)=\bad(a,b,\dv-\eb)$ for which, in turn, we need to estimate 
$\sum \Ygr_{avb}(\dv-\eb)$, where the sum is over all $v\in [n]$ such that both $av$ and $bv$ are allowable (see \eqref{def:bad}). 
By definition~\eqref{Ygrdef}, 
\begin{align}\lab{aux667}
\Ygr_{avb}(\dv-\eb) 
	&=\Pgr_{av}(\dv-\eb) \Pgr_{bv}(\dv-\eb-\ea-\ev)(1+1/n)\nonumber\\
	&= \pi  \cdot \pi(\eps_{b}-\delta,\eps_{v}-\delta ) 
	 \left(1 + \mu\delta  + O\left(\mu\delta^2\right)\right), 
\end{align}
by~\eqref{aux112} and  as $\mu\delta=1/(n-1)$.
As  
$\cA(a)\cap \cA(b)=[n]\setminus \{a,b\}$, defining  $\xi = \mu\delta^2 +  \mu\eps\delta$ this gives 
\begin{align}\lab{aux305}
 & \frac{1}{d_a}\sum_{v\in \cA(a)\cap \cA(b)}\Ygr_{avb}(\dv-\eb) \nonumber\\
	 &\qquad = \frac{1}{d_a}\sum_{\substack{1\leq v\leq n\\ v\neq a,b}}
	 	\pi  \cdot \pi(\eps_{b}-\delta,\eps_{v}-\delta) 
		 \left(1 + \mu\delta  + O\left(\mu\delta^2\right)\right) \nonumber\\
  	&\qquad=   \frac{\mu\left(1+\eps_{b}-\delta\right)}{n-1}
		\left(1 + \mu\delta+ O\left(\xi \right)\right) \sum_{\substack{1\leq v\leq n\\ v\neq a,b}}
		(1+\eps_{v})\left(1+\eps_{v}-\delta\right)
		\left(1-\frac{\eps_{a}\eps_{v}\mu}{1-\mu}\right)
		\left(1-\frac{\eps_{b}\eps_{v}\mu}{1-\mu}\right) 		\nonumber\\
	&\qquad=  \frac{\mu\left(1+\eps_{b}-\delta\right)}{n-1}
		\left(1 + \mu\delta+ O\left(\xi  \right)\right)  \sum_{\substack{1\leq v\leq n\\ v\neq a,b}}
		\bigg( 
		1-\delta + c_1\eps_v + (1+O(\eps\mu))\eps_v^2 +O(\eps^4\mu)
		\bigg),
 \end{align}
   where $c_1=c_1(\delta,\eps_a,\mu)$ is some suitable function independent of $\eps_v$
that satisfies $c_1 = O(1)$. 
Note that $\sum_{v=1}^{n} \eps_v = \sum_{v=1}^{n} (d-d_v)/d = 0$ 
and $\sum_{v=1}^{n} \eps_v^2 =  n \sigma^2/ d^2$, by definition of  $\sigma^2$. 
Hence
\begin{align}
&\left(1 + \mu\delta+ O\left( \xi \right)\right)  \sum_{\substack{1\leq v\leq n\\ v\neq a,b}}
		\bigg( 
		1-\delta + c_1\eps_v + (1+O(\eps\mu))\eps_v^2 +O(\eps^4\mu)
		\bigg)\nonumber\\		
&\qquad = \left(1 + \mu\delta\right) 
		\left( (n-2)(1-\delta) - c_1(\eps_a+\eps_b) + (1+O(\eps\mu))\left( n \sigma^2/d^2-\eps_a^2-\eps_b^2\right)\right)
+ O(n\eps^4\mu+n\xi) .  \nonumber
\end{align}
On the other hand, noting that for~\eqn{def:bad}   in this case $\cA(a)\setminus \cA(b)=\{b\}$ (so the first summation only has one term) and $\pv_{ab} =P^{\mathrm{gr}}_{ab}$, we compute, using~\eqn{aux112},  that
$$\frac{\pv_{ab}(\dv-\eb)}{d_a}=  \frac{ \pi(\eps_{a},\eps_{b}-\delta)}{\mu(n-1)(1+\eps_a)}\left(1+O\left(\mu \delta^2\right)\right)=\frac{1+\eps_b}{n-1}\left(1+O\left( \eps^2\mu +\mu \delta^2\right)\right).
$$ 
Thus, from~\eqn{def:bad},
\begin{align}\lab{aux306}
\bad(a,b,\dv-\eb) 
	&=   \frac{1+\eps_b}{n-1}+ \frac{\mu\left(1+\eps_{b}-\delta\right)}{n-1}\left(1 + \mu\delta\right) 
		\left( (n-1)(1-\delta) - 1 +  n \sigma^2/d^2  \right) \nonumber\\
	&\qquad + O\left(\mu \delta^2+ \frac{\eps\mu}{n} + \eps^4\mu^2\right)\nonumber \\
	&=\mu(1+\eps_b)+ \frac{\sigma^2}{dn}(1+\eps_b ) -\frac{1}{n} 
	+ O\left(\mu \delta^2+ \frac{\eps\mu}{n} + \eps^4\mu^2\right),
\end{align}
where we use that $ \sigma^2 = O(d) $ and that $\delta= 1/d = 1/(\mu (n-1))$ and note some non-trivial cancellations. 
The analogous formula is obtained for  $\bad(b,a,\dv-\ea)$ by swapping indices. 
Hence 
\begin{align*}
 \Rc(\Pgr,\Ygr)_{ab}(\dv) 
 	&=\frac{d_{a}}{d_{b}}\cdot \frac{1- \bad(a,b,\dv-{\bf e_{b}})}{1- \bad(b, a, {\bf d} - {\bf e_{a}})}\\
	&=\frac{1+\eps_{a}}{1+\eps_{b}}\cdot \frac{1- \mu (1+\eps_{b})}{1-  \mu(1+\eps_{a}) }
		\cdot\left(1+(\eps_{a}-\eps_{b})  \frac{\sigma^2}{dn} +O\left(\mu \delta^2+ \frac{\eps\mu}{n} + \eps^4\mu^2\right)\right)\\
	&= R^{\mathrm{gr}}_{ab}(\dv)  + O\left(\mu \delta^2+ \frac{\eps\mu}{n} + \eps^4\mu^2 \right).  
\end{align*}
This proves part \eqref{R-appr} of the lemma.


To prove part \eqref{P-appr} note that, analogous to~\eqn{aux112} we also have for $a, b\neq v$
$$ R^{\mathrm{gr}}_{ab}(\dv-\ev)  = \rho(\eps_{a}',\eps_{b}',\mu',(\sigma^2)',d') 
=\rho\cdot\left(1 +O\left( \mu \delta^2  \right)\right).$$
(In particular, $(\sigma^2)' -  \sigma^2=  O(\max |d_i-d|/n)=O(\eps d/n)$.) 
Therefore, by definition \eqref{F1def},  
\begin{align}\lab{aux155}
\Pc(P^{\mathrm{gr}}, R^{\mathrm{gr}})_{av}(\dv) 
	&= d_{v} \left(\sum_{b\in [n]\setminus \{v\}}
		R^{\mathrm{gr}}_{ba} (\dv-\ev)
		\frac{1- P^{\mathrm{gr}}_{bv}(\dv  - \eb - \ev)}
		{1-P^{\mathrm{gr}}_{av}( \dv -\ea - \ev)}\right)^{-1}\nonumber\\
	&= d_{v} \left(\sum_{b\in [n]\setminus \{v\}} 
		\rho\cdot
		\frac{1-\pi\left(\eps_{b}-\delta, \eps_{v}-\delta \right)}
		{1-\pi\left(\eps_{a}-\delta, \eps_{v}-\delta\right)}
		\left(1 +O\left(\mu \delta^2 \right)\right)\right)^{-1}\nonumber\\
	&= d_{v}(1+\eps_{a})\frac{1-\pi\left(\eps_{a}-\delta,\eps_{v}-\delta \right)}{1-\mu(1+ \eps_{a})}
	\left( 1 +O\left(\frac{\eps}{n}+\mu \delta^2\right)\right)\nonumber\\
	&\qquad \times 
	\left(\sum_{b\in  [n]\setminus \{v\}} 		
		(1+\eps_{b})\frac{1-\pi\left(\eps_{b}-\delta,\eps_{v}-\delta \right)}{1-\mu(1+\eps_{b})}\right)^{-1}, 
\end{align}
where we used $ \sigma^2=O(d)$ and $\pi(\cdot,\cdot) = O(\mu)$. 
Straightforward calculations show that for $c\in\{a,b\}$, 
\begin{align}\lab{aux156}
\frac{1-\pi\left(\eps_{c}-\delta,\eps_{v}-\delta \right)}
	{1-\mu\left(1+\eps_{c}\right)}
	& = 1+ \frac{\mu}{1-\mu} 
	\left( 2\delta 
		-\eps_v - \eps_v\eps_c 
	\right)
			+O\left(\frac{\eps}{n}+\eps^3\mu^2+\mu\delta^2\right)\nonumber\\
	& = \left(1- \frac{\mu\eps_v\eps_c}{1-\mu} \right)
	\left(1-  \frac{\mu(\eps_v-2\delta)}{1-\mu}
		+ O\left(\frac{\eps}{n}+\eps^3\mu^2+\mu\delta^2\right)\right).		
\end{align}
Therefore, \eqref{aux155}  is equivalent to 
\begin{align*}
\Pc(P^{\mathrm{gr}}, R^{\mathrm{gr}})_{av}(\dv)
	&=d_{v}(1+\eps_{a})\left(1-\frac{\eps_{a}\eps_{v}\mu}{1-\mu}\right) 
		\left(1+O\left(\frac{\eps}{n}+ \mu \delta^2+\eps^3\mu^2\right)\right)
	\sum_{\substack{1\leq b\leq n\\ b\neq v}} 
	( 1+\eps_{b} )  \left(1-\frac{\eps_{b}\eps_{v}\mu}{1-\mu}\right)\nonumber\\
	&=\mu(1+\eps_{v})(1+\eps_{a})\left(1-\frac{\eps_{a}\eps_{v}\mu}{1-\mu}\right) 
		\left(1+O\left(\frac{\eps}{n}+ \mu \delta^2+\eps^3\mu^2\right)\right)\nonumber\\
	&=\pi\cdot \left(1+O\left(\frac{\eps}{n}+ \mu \delta^2+\eps^3\mu^2\right)\right), 
\end{align*}
where in the second equality we use $\sum_{b} \eps_{b} = 0$ and $\sum_{b} \eps_{b}^2 = O\left(1/\mu\right)$.

For part \eqref{Q-appr} we have by definition of $\Yc$ in~\eqref{FQdef}
\begin{align}\lab{q-aux381} 
\c(\Pgr,\Ygr)_{avb}(\dv) 
	&= \frac{P^{\mathrm{gr}}_{av}(\dv)\left(P^{\mathrm{gr}}_{bv}(\dv-\ea-\ev) - \Ygr_{avb}(\dv-\ea-\ev)\right)}{1-P^{\mathrm{gr}}_{av}(\dv-\ea-\ev)} 
\end{align}
for distinct $a,v,b\in [n]$. 
By definition of $\Ygr$ in~\eqref{Ygrdef}, and~\eqn{aux112},  we obtain as in~\eqn{aux667}
\begin{align*}
\Ygr_{avb}(\dv-\ea-\ev) 
	&= P^{\mathrm{gr}}_{av}(\dv-\ea-\ev) P^{\mathrm{gr}}_{bv}(\dv-2\ea-2\ev) (1+1/n)\nonumber\\
	&= \pi(\eps_{a}-\delta,\eps_{v}-\delta) \cdot \pi(\eps_{b},\eps_{v}-2\delta) (1+\mu\delta+O(\xi))\nonumber\\
&= \pi(\eps_{b},\eps_{v}-\delta) \cdot \pi(\eps_{a}-\delta,\eps_{v}-2\delta) (1+\mu\delta+O(\xi))\nonumber\\	
	&= P^{\mathrm{gr}}_{bv}(\dv-\ea-\ev) P^{\mathrm{gr}}_{av}(\dv-\ea-2\ev) (1+\mu\delta+O(\xi)).
\end{align*}
Plugging this into~\eqn{q-aux381} we obtain
\begin{align*}
\Yc(\Pgr,\Ygr)_{avb}(\dv) 
	&= P^{\mathrm{gr}}_{av}(\dv)P^{\mathrm{gr}}_{bv}(\dv-\ea-\ev)
	\frac{1-P^{\mathrm{gr}}_{av}(\dv-\ea-2\ev) (1+\mu\delta+O(\xi))}{1-P^{\mathrm{gr}}_{av}(\dv-\ea-\ev)}. 	 
\end{align*}
Now straight-forward calculations show that 
\begin{align*} 
\frac{1-P^{\mathrm{gr}}_{av}(\dv-\ea-2\ev) (1+\mu\delta+O(\xi))}{1-P^{\mathrm{gr}}_{av}(\dv-\ea-\ev)} 
	&= \frac{1-\mu(1+\eps_a-\delta)(1+\eps_v-2\delta)\left(1+\frac{\eps_a\eps_v\mu}{1-\mu}+\mu\delta+O(\xi)\right)}
	{1-\mu(1+\eps_a-\delta)(1+\eps_v-\delta)\left(1+\frac{\eps_a\eps_v\mu}{1-\mu}+O(\xi)\right)}\\
	&= \frac{1-B +\delta\mu-\mu/n +O(\xi)}{1-B}\\
	&= 1+\mu\delta+ O(\xi), 
\end{align*}
where 
$$B = \mu(1+\eps_a-\delta)(1+\eps_v-\delta)\left(1+\frac{\eps_a\eps_v\mu}{1-\mu}\right)
=\mu + O(\mu \eps).$$
Thus it follows that 
$\Yc(\Pgr,\Ygr)_{avb}(\dv)
= \Ygr_{avb}(\dv)(1+O(\xi))$ and we are done.  
\end{proof}

 \end{document}